\documentclass[11 pt,reqno]{amsart}

\usepackage{amsmath,amsthm,amssymb,amscd,graphicx}
\usepackage{bbm, dsfont}
\usepackage{color}
\usepackage{hyperref}
\usepackage{enumerate}
\usepackage{mathtools}
\usepackage{mathrsfs}
\usepackage{multirow}

\newtheorem*{acknowledgements}{Acknowledgements}
\newtheorem{thm}{Theorem}[section]
 \newtheorem{cor}[thm]{Corollary}
 \newtheorem{lem}[thm]{Lemma}
 
 \newtheorem{prop}[thm]{Proposition}
 
 \theoremstyle{definition}
 
 \theoremstyle{remark}
 \newtheorem{rem}[thm]{Remark}
 
 \numberwithin{equation}{section}

\def\be#1 {\begin{equation} \label{#1}}
\newcommand{\ee}{\end{equation}}
\renewcommand{\phi}{\varphi}

\def\C{\mathbb C}
\def\R{\mathbb R}
\def\T{\mathbb T}

\def\HH{\mathbb H}
\def\N{\mathbb N}
\def\E{\mathcal E}

\def\pa{\partial}

\def\dis{\displaystyle}    

\definecolor{gr}{rgb}   {0.,   0.69,   0.23 }
\definecolor{bl}{rgb}   {0.,   0.5,   1. }
\definecolor{mg}{rgb}   {0.85,  0.,    0.85}
\definecolor{yl}{rgb}   {0.8,  0.7,   0.}
\definecolor{or}{rgb}  {0.7,0.2,0.2}
\newcommand{\Bk}{\color{black}}

\newcommand{\wt}{\widetilde}

\catcode`\ = \active\def {\'e} \catcode`\ = \active\def {\'E}  \catcode`\Ë = \active\def Ë{\`A}\catcode`\ = \active\def {\`e}
\catcode`\ = \active\def {\c{c}} \catcode`\ = \active\def
{\`a} \catcode`\ = \active \def {\^e} \catcode`\ = \active
\def {\`u} \catcode`\ = \active \def {\^o} \catcode`\ =
\active \def {\^u} \catcode`\ = \active \def {\^{\i}}
\catcode`\ = \active \def {\^a} \catcode`\ = \active \def
{{\"o}}

\renewcommand{\Re}{  {\mathfrak{Re}}  }
 \renewcommand{\Im}{   {\mathfrak{Im}} }
\newcommand{\ov}{  \overline  }

\newcommand\<{\langle}
\renewcommand\>{\rangle}

\begin{document}

\thanks{V. Schwinte is  supported by the grant   "ISDEEC'' ANR-16-CE40-0013}
 
\thanks{L. Thomann is supported by the grants   "BEKAM''  ANR-15-CE40-0001 and  "ISDEEC'' ANR-16-CE40-0013}

\author{Valentin Schwinte}
\address{Universit de Lorraine, Mines Nancy, F-54000 Nancy, France}
\email{valentin.schwinte1@etu.univ-lorraine.fr}
\author{ Laurent Thomann }
\address{Universit de Lorraine, CNRS, IECL, F-54000 Nancy, France}

\email{laurent.thomann@univ-lorraine.fr}

\title{Growth of Sobolev norms for coupled Lowest Landau Level equations}



\begin{abstract}
We study coupled systems of nonlinear lowest Landau level  equations, for which we prove global existence results with polynomial bounds on the possible growth of Sobolev norms of the solutions. We also exhibit explicit unbounded trajectories which show that these bounds are optimal. \end{abstract}

\maketitle

%

\section{Introduction and main results}
 
 \subsection{General introduction}

 A Bose-Einstein condensate is a state of matter obtained by decreasing the temperature of a gas made of bosons without potential energy, which are all identical, with a spin that equals $0$, and which are non relativistic. If the  temperature is below   a critical temperature (named Bose-Temperature), almost all the bosons are in the lowest quantum state, and this  phenomenon is called Bose-Einstein condensation. 
 
Consider a Bose-Einstein condensate confined by a harmonic field, and rotating at a high velocity. Then its dynamics can be described by	the Lowest Landau Level (LLL) equation
   \begin{equation}\label{eq-LLL} 
\left\{
\begin{aligned}
&i \partial_t u=\Pi(|u|^2u), \quad   (t,z)\in \R\times \C,\\
&u(0,\cdot)=  u_0 \in \E,
\end{aligned}
\right.
\end{equation}
where $\E$ is the   Bargmann-Fock space defined as 
$$
\mathcal{E} =\big \{\, u(z) = e^{-\frac{|z|^2}{2}} f(z)\,,\;f \; \mbox{entire\ holomorphic}\,\big\}\cap L^2(\C ),
$$
and $\Pi$   is the  orthogonal projection on $\mathcal{E}$. For more details and references on this modeling, we refer to~\cite{ABD, Ho}, and the introduction of \cite{GGT}.

In the case of  two-components Bose-Einstein condensates, \textit{i.e.} two coupled Bose-Einstein condensates, the corresponding equation reads \cite{Mueller-Ho}
\begin{equation}\label{sys-sig-0} 
    \left\{
        \begin{array}{ll}
        i\partial_tu = \alpha \Pi(|u|^2u) + \beta \Pi(|v|^2u), &(t,z) \in \R \times \C, \\
        i\partial_tv = \gamma \Pi(|v|^2v) + \beta \Pi(|u|^2v), &\\
        u(0,\cdot) = u_0 \in \E,\quad v(0,\cdot)=v_0\in \E,&
        \end{array}
    \right.
\end{equation}
with $\alpha, \beta, \gamma \in \R$.   Coupled Bose-Einstein condensates also have applications in superfluidity and superconductivity.

\subsection{Mathematical motivation}

In this paper, we will  focus on the dynamics of the  following model system
 \begin{equation}\label{sys-sig} 
\left\{
\begin{aligned}
&i\partial_{t}u= \Pi (|v|^2 u), \quad   (t,z)\in \R\times \C,\\
&i\partial_{t}v=  \sigma \Pi (|u|^2 v),\\
&u(0,\cdot)=  u_0 \in\E,\; v(0,\cdot)=  v_0 \in\E,
\end{aligned}
\right.
\end{equation}
where $\sigma \in \{1,-1\}$. Some of our results, which rely on general arguments, can easily be extended to the complete system~\eqref{sys-sig-0}, while other parts of our work rely on explicit computations and depend heavily on the coefficients of the system.     The system \eqref{sys-sig} is Hamiltonian with the structure
\begin{equation*} 
\left\{
\begin{aligned}
&\dot u=-i\frac{\delta \mathcal{H}}{\delta \ov u},& \quad &\dot{\ov u}=i\frac{\delta \mathcal{H}}{\delta u},\\
&\dot v=-i\sigma  \frac{\delta \mathcal{H}}{\delta \ov v},& \quad &\dot{\ov v}=i\sigma \frac{\delta \mathcal{H}}{\delta v},
\end{aligned}
\right.
\end{equation*} 
and  its Hamiltonian functional reads 
$$\mathcal{H}(u,v)=\int_{\C}|u|^2|v|^2dL,$$
where $L$ stands for Lebesgue measure on $\C$. We will see that the qualitative dynamics of~\eqref{sys-sig} crucially depend on the  sign of $\sigma$. In the physical modeling, $\sigma=1$, but from a mathematical point of view we will see that it is  interesting to consider the case $\sigma=-1$ because new phenomena will occur.  \medskip

As we have mentioned above, equation \eqref{eq-LLL} is a model for fast rotating Bose-Einstein condensate. This equation can also be derived as a big box limit for weakly non-linear Schr\"odinger equations~\cite{FGH} or a longtime limit of a Gross-Pitaevskii equation with partial confinement \cite{HT}. We refer to the introduction of \cite{GGT} for more details. It is likely that similar derivation results may be obtained for~\eqref{sys-sig}, in the case $\sigma=1$ (in this latter case the Sobolev norm $\HH^1(\C)$ is a conservation law of the system, see below).\medskip

Several papers \cite{Nier, ABN, BiBiCrEv2, GGT} were devoted to the study of the  LLL equation \eqref{eq-LLL} in which dynamical aspects (well-posedness, bounds on the solutions\dots) as well as  stationary solutions  (classification of stationary solutions with finite number of zeros,  growth, stability results\dots) were studied.  It turns out that most of the tools developed in \cite{GGT} can apply to \eqref{sys-sig} and give similar results, at least for $\sigma=1$.    Equation \eqref{eq-LLL} is globally well-posed on $\mathcal{E}$ and we will see that it is also  the case of equation \eqref{sys-sig} (see Theorem~\ref{mainCauchy} below). We refer to \cite{BiBiCrEv, BiBiEv, Clerck-Evnin} for more results on LLL and related equations. \medskip

A natural question is the large time  description of the dynamics of the solutions to \eqref{eq-LLL} and~\eqref{sys-sig} and the behaviour of the    Sobolev norms.  A growth of Sobolev norm corresponds to a  transfer of energy from low to high frequencies, but in the Bargmann-Fock space $\E$, this is  equivalent to   a transfer in the physical space, since we have the following characterization of the Sobolev spaces in  $\E$ (see~\cite[Lemma~C.1]{GGT}): let $s \geq 0$, then there exist $c,C>0$ such that for all  $u\in  \HH^s(\C)$ (where~$\HH^s(\C)$ stands for the $L^2(\C)$-based Sobolev space, adapted to this context, see \eqref{def-sobo})
\begin{equation}\label{eq-hs}
c\|\<z\>^s u\|_{L^2(\C)} \leq  \|u\|_{\HH^s(\C)} \leq C\|\<z\>^s u\|_{L^2(\C)}, \quad \<z\>=(1+|z|^2)^{1/2}.
\end{equation}
 We are able to exhibit such a phenomenon of transfer of energy, but only for \eqref{sys-sig} in the case $\sigma=-1$ (the question about a possible growth of norms for \eqref{eq-LLL} and \eqref{sys-sig} with $\sigma=1$ is left open). This will be achieved thanks to the construction of explicit progressive waves (also called traveling waves) using the magnetic translations which are symmetries of the equation~\eqref{sys-sig}. Indeed, by~\eqref{eq-hs}, we see that non trivial  progressive waves have growing Sobolev norms (the norm $\HH^s(\C)$ defined above is not translation invariant, contrary to the usual Sobolev space $H^s(\R^n)$).

\subsection{Symmetries and conservation laws}
Observe that the system \eqref{sys-sig} is left invariant by several symmetries, which induce conservation laws (we refer to \cite[Section 2]{GGT} for more details). These symmetries are phase rotations
\begin{equation*} 
T_{\theta_1,\theta_2} : (u,v)(z) \mapsto \big(e^{i \theta_1 }u(z), e^{i \theta_2 }v(z)\big)  \qquad \mbox{for $(\theta_1, \theta_2)  \in \mathbb{T}^2$},
\end{equation*}
 space rotations 
\begin{equation*} 
L_{\theta} :(u,v)(z)  \mapsto \big(u(e^{i\theta}z), v(e^{i\theta}z) \big)  \qquad \mbox{for $\theta \in \mathbb{T}$},
\end{equation*}
and magnetic translations 
\begin{equation}\label{defR}
R_{\alpha} :(u,v)(z)  \mapsto \big(u(z+\alpha) e^{\frac{1}{2}(\overline z \alpha - z \overline{\alpha})}, v(z+\alpha) e^{\frac{1}{2}(\overline z \alpha - z \overline{\alpha})} \big) \qquad \mbox{for $\alpha \in \mathbb{C}$}.
\end{equation}
The corresponding conservation laws are: the mass
$$ M(u)=\int_{\C}|u(z)|^2dL(z), \quad M(v)=\int_{\C}|v(z)|^2dL(z),$$
  the angular momentum
    \begin{equation*}
 P_\sigma(u,v)=\int_{\C}\big(|z|^2-1\big)\big(|u(z)|^2+\sigma|v(z)|^2\big)dL(z),
 \end{equation*}
and the magnetic momentum
   \begin{equation*}
 Q_\sigma(u,v)=\int_{\C}z\big(|u(z)|^2+\sigma|v(z)|^2\big)dL(z).
 \end{equation*}
In the sequel, for short, we write $ P_{+}=P_1$,   $ P_{-}=P_{-1}$ and $ Q_{+}=Q_1$,   $ Q_{-}=Q_{-1}$. \medskip

We are now ready to state our main results:

\subsection{Global existence results for the nonlinear system}

To begin with, let us state a global well-posedness result which holds true for both cases $\sigma \in \{1,-1\}$.

\begin{thm}\label{mainCauchy}
For every $(u_0,v_0)\in \mathcal{E} \times \mathcal{E}$, there exists a unique solution $(u,v)\in \mathcal{C}^{\infty}  (\R ,\mathcal{E}\times \mathcal{E})$ to the system  \eqref{sys-sig}, and this solution depends smoothly on $(u_0,v_0)$. Moreover, for every  $t\in \R$ 
$$ M(u)=\int_{\C}|u(t,z)|^2dL(z)=M(u_0), \quad M(v)=\int_{\C}|v(t,z)|^2dL(z)=M(v_0),$$
and 
 $$\mathcal{H}(u,v)=  \int_{\C}|u(t,z)|^2|v(t,z)|^2dL(z)=\mathcal{H}(u_0,v_0).$$
  Furthermore, if $(zu_0,zv_0)\in L^2(\C )\times  L^2(\C )$, then  $\big(zu(t), zv(t)\big)\in L^2(\C)\times  L^2(\C )$ for every $t\in \R$, and 
      \begin{equation*}
 P_\sigma(u,v)=\int_{\C}\big(|z|^2-1\big)\big(|u(t,z)|^2+\sigma|v(t,z)|^2\big)dL(z)= P_\sigma(u_0,v_0),
 \end{equation*}
   \begin{equation*}
 Q_\sigma(u,v)=\int_{\C}z\big(|u(t,z)|^2+\sigma|v(t,z)|^2\big)dL(z)= Q_\sigma(u_0,v_0).
 \end{equation*}
 More generally, if for some $s>0$, $\big(\langle z\rangle ^s u_0, \<z\rangle ^s v_0\big)\in L^2(\C)  \times  L^2(\C ) $, then  $\big(\langle z\rangle ^su(t), \langle z\rangle ^sv(t)\big) \in L^2(\C) \times  L^2(\C ) $ for every $t\in \R$.
\end{thm}
In the previous statement let us stress that in the case $\sigma=1$, $M+P_+$ corresponds to the square of the $\HH^1$ norm, thus this norm is conserved, while in the case $\sigma=-1$, even if $P_-$ is preserved one could have 
$$\int_{\C}|z|^2 |u(t,z)|^2 dL(z) \longrightarrow +\infty,\quad \int_{\C}|z|^2 |v(t,z)|^2 dL(z) \longrightarrow +\infty, \quad t \longrightarrow \pm \infty,
  $$
 and this is actually what we will prove in Corollary \ref{cor-growth}. \medskip
 
We end this paragraph by stating  a result which shows that  the $L^{\infty}$ norm of any solution to~\eqref{sys-sig} is essentially constant. In particular, this shows that the solution to \eqref{sys-sig} never disperses.

\begin{prop}\label{propLinfini}
Let  $(u_0,v_0)\in \mathcal{E} \times \mathcal{E}$ and consider  $(u,v)\in \mathcal{C}^{\infty}  (\R ,\mathcal{E}\times \mathcal{E})$  the solution  to the system~\eqref{sys-sig}. Then  for all~$t \in \R$
 \begin{equation*} 
\begin{aligned}
& \frac{\sqrt{\pi}\mathcal{H}(u_0,v_0)}{\|u_0\|_{L^2(\C)}\|v_0\|^2_{L^2(\C)}} \leq  \|u(t)\|_{L^\infty(\C)}\leq \frac{1}{\sqrt{\pi}}\|u_0\|_{L^2(\C)},  \\[8pt]
& \frac{\sqrt{\pi}\mathcal{H}(u_0,v_0)}{\|v_0\|_{L^2(\C)}\|u_0\|^2_{L^2(\C)}} \leq  \|v(t)\|_{L^\infty(\C)}\leq \frac{1}{\sqrt{\pi}}\|v_0\|_{L^2(\C)}.
\end{aligned}
\end{equation*}
\end{prop}

This is a  rigidity  result which is induced by the properties of the space $\E$ and by the conservation laws of \eqref{sys-sig} and does not rely on the specific dynamics of~\eqref{sys-sig}. We refer to~\cite[Lemma 3.3]{Caze} for a similar property for the Schr\"odinger equation with logarithmic nonlinearity.

\subsection{Progressive  waves and growth of Sobolev norms}
In view of the   invariances induced  by phase rotations and magnetic translations, it is natural to define progressive  waves for equation~\eqref{sys-sig} as solutions of the form
 \begin{equation}\label{ansatz-1}
 \big(u(t,z), v(t,z)\big)=\big(e^{-i\lambda t}U(z+\alpha t) e^{\frac{1}{2}(\overline z \alpha - z \overline{\alpha})t},e^{-i \mu t}V(z+\alpha t) e^{\frac{1}{2}(\overline z \alpha - z \overline{\alpha})t} \big),
 \end{equation}
for some $(\lambda ,\mu )\in \R^2$ and $\alpha =\alpha_1+i \alpha_2 \in \C$. Equivalently, the corresponding initial condition  $(U,V)$  have to satisfy the system
 \begin{equation} \label{syst}
\left\{
\begin{aligned}
& \lambda U+ (\alpha \cdot  \Gamma) U = \Pi (|V|^2 U),  \\
& \mu V+ (\alpha \cdot  \Gamma) V = \sigma \Pi (|U|^2 V),
\end{aligned}
\right.
\end{equation}
where $\alpha \cdot  \Gamma  :=\alpha_1 \Gamma_1 +\alpha_2 \Gamma_2$   (see Appendix~\ref{appendix})    and 
\begin{equation*} 
\Gamma_1=i(-z +\partial_z+\frac{\ov z}2), \quad \Gamma_2=-(z +\partial_z+\frac{\ov z}2).
\end{equation*}
 
 To be begin with, let us see if the existence of progressive waves is compatible with the conservation laws. Actually, we have the following relations, which hold true for all $\beta \in \C$
 \begin{equation}\label{rot-Q}
Q_\sigma(R_\beta U, R_\beta V)=Q_\sigma(U,V)-\beta  \big(M(U)+\sigma M(V)\big)\,,
\end{equation} 
and
\begin{equation} \label{rot-P}
P_\sigma(R_\beta U, R_\beta V)= P_\sigma(  U,  V) -\big(\beta \ov{Q_\sigma(U,V)} +\ov{\beta} {Q_\sigma(U,V)} \big)+|\beta|^2\big(M(U)+\sigma M(V)\big),
\end{equation}
 and this suggests that the sign of  $\sigma \in \{1,-1\}$ will play a role in the existence of traveling waves. \medskip

For $s\geq 0$, denote by $L^{2,s}=\big\{u\in \mathscr{S}'(\C), \;\<z\>^su \in L^2(\C)\big\}$ and $L^{2,s}_{\E}=L^{2,s}  \cap \E$.

\subsubsection{Case $\sigma=1$} In the case $\sigma =1$, we obtain   similar qualitative results to the one obtained for the LLL equation  $i\partial_t u=\Pi(|u|^2u)$, and the arguments are similar.\medskip

To begin with, we can prove the following global existence result, with  polynomial bounds on the Sobolev norm of general solutions to~\eqref{sys-sig}:
 
 \begin{thm} \label{borne-poly}
Assume that $\sigma=1$.  Let $k\geq 0$ be an integer and $(u_0,v_0)\in L^{2,k}_\mathcal{E} \times L^{2,k}_\mathcal{E}$. Then there exists a unique solution $(u,v)\in \mathcal{C}^{\infty}  \big(\R , L^{2,k}_\mathcal{E} \times L^{2,k}_\mathcal{E}\big)$ to  equation~\eqref{sys-sig} and it satisfies for all $t \in \R$,
\begin{equation}\label{bk3}
   \begin{aligned}  
&\Vert \langle z\rangle ^ku(t)\Vert _{L^2(\C)}+\Vert \langle z\rangle ^kv(t)\Vert _{L^2(\C)} \lesssim  \big(1+|t| \big)^{\frac{k-1}4} &\quad\text{if} &\quad k\geq 3\\
&\Vert \langle z\rangle ^2u(t)\Vert _{L^2(\C)}+\Vert \langle z\rangle ^2v(t)\Vert _{L^2(\C)}\lesssim  \big(1+|t| \big)^{\frac12}        &\quad\text{if} &\quad k=2.
\end{aligned}
\end{equation}
 Moreover, if  $\langle z\rangle^3 u_0\in L^2(\C)$ and $\langle z\rangle^3 v_0\in L^2(\C)$, then 
 \begin{equation}\label{newB2}
 \Vert \langle z\rangle ^2u(t)\Vert _{L^2(\C)}+\Vert \langle z\rangle ^2v(t)\Vert _{L^2(\C)} \lesssim  \big(1+|t| \big)^{\frac14}.
 \end{equation}
\end{thm}

 The equivalence of norms \eqref{eq-hs} states that, roughly speaking, a weight $z$ is equivalent to a  derivative $\partial_z$, in the $L^2$-norm.    The key observation used in the proof of Theorem~\ref{borne-poly} is that for any $u \in \E$, we have (see Lemma~\ref{lem.deri} below), for all $j,k \in \N$
$$  \big\| \partial^j_{\ov z} \partial^k_z\big(|u|^2\big)\big\|_{L^\infty(\C)} \leq C_{jk}  \| u\|^2_{L^\infty(\C)},$$
which allow to control high order derivatives, without using Sobolev norms. \medskip

   The  same argument can be used for the LLL equation, and this allows to improve the bounds obtained in~\cite[Theorem~1.2]{GGT}, we have added the statement in the Appendix~\ref{AppendB}. \medskip

The constants in Theorem~\ref{borne-poly} can be made more precise. Denote by 
$$X_k(t)=\Vert \langle z\rangle ^ku(t)\Vert _{L^2(\C)}+\Vert \langle z\rangle ^kv(t)\Vert _{L^2(\C)},$$ 
then we actually prove   (see~\eqref{borne-expl}) that in any of the cases \eqref{bk3},
\begin{equation*} 
   \begin{aligned}  
&X_k(t) \leq C X_k(0) \big(1+X^2_1(0) |t| \big)^{\frac{k-1}4}&\quad\text{if} &\quad k\geq 3\\
&X_2(t) \leq C X_2(0) \big(1+X^2_1(0) |t| \big)^{\frac{1}2}        &\quad\text{if} &\quad k=2,
\end{aligned}
\end{equation*}
but in all cases, we do not think that the powers of $t$ in the bounds are optimal. \medskip

Our next result shows,  as for LLL, that there are no nontrivial progressive waves in $L^{2,1}_\mathcal{E}$ for \eqref{sys-sig}: 
 
\begin{prop} 
Assume that $\sigma=1$, then there exist no progressive waves in $L^{2,1}_\mathcal{E}$ with $\alpha \neq 0$ to the system~\eqref{sys-sig}.
\end{prop}
 
The proof of this statement is very simple: we apply   \eqref{rot-Q} with $\beta=\alpha t$, and this gives a contradiction with the fact that  $Q_+$  is a conservation law of the system.

\subsubsection{Case $\sigma=-1$} In the case $\sigma=-1$, our result on the control of the norms for \eqref{sys-sig} reads as follows:

 \begin{thm}\label{thm1.6}
Assume that $\sigma=-1$.  Let $k\geq 0$ be an integer and $(u_0,v_0)\in L^{2,k}_\mathcal{E} \times L^{2,k}_\mathcal{E}$. Then there exists a unique solution $(u,v)\in \mathcal{C}^{\infty}  \big(\R , L^{2,k}_\mathcal{E} \times L^{2,k}_\mathcal{E}\big)$ to  equation~\eqref{sys-sig} and it satisfies for all~$t \in \R$,
\begin{equation}\label{borne-opt}
\begin{aligned} 
&  \big\| \<z\>^ku(t)\big\|_{L^2(\C)}  \leq   \big\| \<z\>^ku_0\big\|_{L^2(\C)} \big(1+C \|v_0\|^2_{L^2(\C)}  |t| \big)^k  \\
&    \big\| \<z\>^kv(t)\big\|_{L^2(\C)}  \leq   \big\| \<z\>^kv_0\big\|_{L^2(\C)} \big(1+C \|u_0\|^2_{L^2(\C)}  |t| \big)^k.
\end{aligned}
\end{equation}
Moreover,  for all $t \in \R$,
\begin{equation}\label{borneck}
\begin{aligned} 
&  \Big|   \big\|  \langle z\rangle ^ku(t)  \big\| ^2 _{L^2(\C)}-  \big\|  \langle z\rangle ^kv(t)  \big\| ^2 _{L^2(\C)} \Big| \lesssim  \big(1+|t| \big)^{2k-1}&\quad\text{if} &\quad k\geq 2. \\
&   \Big|   \big\|  \langle z\rangle ^ku(t)  \big\| ^2 _{L^2(\C)}-  \big\|  \langle z\rangle ^kv(t)  \big\| ^2 _{L^2(\C)} \Big| \lesssim  1&\quad\text{if} &\quad k=0,1.
\end{aligned}
\end{equation}

\end{thm}

 In the case $\sigma=-1$, the   $L^{2,1}$ norm is no more controlled thus  the bounds on the solutions are cruder. However, we will see in  
Corollary~\ref{cor-growth} that the bounds \eqref{borne-opt} are optimal. \medskip

Concerning the progressive waves, in the case $\sigma=-1$,  the relations \eqref{rot-Q} and \eqref{rot-P} do no more  give an insurmountable obstruction to their existence. By taking $\beta=\alpha t$, we see that any   non trivial progressive wave has to satisfy $M(U)=M(V)$ and $\ov{\alpha} Q_-(U,V) \in i \R$. Indeed, we can prove that such solutions exist, and we are able to classify the ones which have a finite number of zeros.\medskip

For $n\geq 0$, we define the following family of $L^2(\C)$-normalized functions of $\E$
$$
\varphi_n(z) = \frac{1}{\sqrt{\pi n!}} z^n e^{-\frac{|z|^2}{2}},
$$
which forms a Hilbertian basis of $\E$, and for $\gamma \in \C$, we define 
\begin{equation*} 
\phi_n^\gamma(z) = R_{- \overline \gamma} (\phi_n)(z) = \frac{1}{\sqrt{\pi n!}} (z-\overline \gamma)^n e^{-\frac{|z|^2}{2}-\frac{|\gamma|^2}{2} + \gamma z}.
\end{equation*}
Then our classification result reads

\begin{thm}\label{classi}
Assume that $\sigma=-1$, then the progressive waves in $\mathcal{E}$ to \eqref{sys-sig} which have a finite number of zeros are given by the initial conditions 
 \begin{enumerate}[$(i)$]
\item when $\alpha=0$
 \begin{equation*}
\left\{
\begin{aligned}
& U= Ae^{i\theta_1}\phi_{n_1}^{\gamma}   \\
& V= Be^{i\theta_2}\phi_{n_2}^{\gamma},
\end{aligned}
\right.
\end{equation*}
 with $A,B \geq 0$, $n_1,n_2 \in \N$, $\theta_1, \theta_2 \in  \R$, $\gamma \in \C$,  where
$$ \lambda = \frac{(n_1+n_2)!}{2^{n_1+n_2+1}\pi n_1! n_2!}B^2, \quad \mu= - \frac{(n_1+n_2)!}{2^{n_1+n_2+1}\pi n_1! n_2!}A^2\;; $$

\item when $\alpha\neq 0$
 \begin{equation}\label{prog}
\left\{
\begin{aligned}
& U= Ke^{i\theta_1} \big(\frac12 \phi_0^\gamma +\frac{\sqrt{3}}2 ie^{i\theta} \phi_{1}^\gamma      \big)   \\
& V= Ke^{i\theta_2} \big(\frac12 \phi_0^\gamma -\frac{\sqrt{3}}2i e^{i\theta} \phi_{1}^\gamma      \big),
\end{aligned}
\right.
\end{equation}
with $K \geq  0$, $\theta, \theta_1, \theta_2 \in  \R$, $\gamma \in \C$, where
$$ \lambda = \frac{K^2}{32\pi}(7+2\sqrt{3}\Im\big(\gamma e^{-i\theta})\big), \quad \mu = \frac{K^2}{32\pi}\big(-7+2\sqrt{3}\Im(\gamma e^{-i\theta})\big), \quad \alpha = \frac{\sqrt{3}}{32\pi}K^2e^{-i\theta}.$$
 \end{enumerate}
\end{thm}
 
  This classification result is in the spirit of \cite[Theorem 6.1]{GGT}. The proof relies on the crucial fact, that an entire function which has a finite number of zeros is a polynomial multiplied by the exponential of an entire function. By adding the fact that we are working in $L^2(\C)$, we obtain a very precise Ansatz and we are able to solve the corresponding system.\medskip
  
The progressive waves~\eqref{prog} moreover satisfy
  \begin{equation*}
  M(U)=M(V)=K^2, \quad \mathcal{H}(U,V)=\frac{11}{64 \pi}K^4,
  \end{equation*}
  \begin{equation}\label{eq00}
Q_-(U,V)= -\frac{\sqrt{3}}2i e^{-i\theta}K^2, \quad P_-(U,V)=\sqrt{3}\Im(\gamma e^{-i\theta})K^2.
\end{equation}

Notice that the speed $\alpha \in \C$ of each traveling wave in Theorem~\ref{classi} is proportional to the square of its size (proportional to its mass).\medskip  

  It is interesting to observe that the  traveling waves defined by~\eqref{prog} have a Gaussian decay, which is not common for a Schr\"odinger-like equation (usually the rate of decay is at most exponential, see e.~g.~\cite{Martel}). We can however mention the   Schr\"odinger equation with logarithmic nonlinearity (logNLS) which has Gaussian solitons  \cite{Caze, Ardila, Car-Galla, Fer1, Fer2} and which possesses several dynamical similarities with~\eqref{sys-sig}. \medskip 
  
 The initial conditions~\eqref{prog} provide explicit examples of solutions to \eqref{sys-sig} with polynomial growth of Sobolev norms and allow to prove that  the bounds of Theorem \ref{thm1.6} are   optimal:
 
 \begin{cor}\label{cor-growth}
 Assume that $(U,V)$ takes the form \eqref{prog} with $K \neq 0$, then corresponding solution $(u,v)$ to~\eqref{sys-sig} satisfies for all $s\geq 0$
 \begin{equation} \label{rate}
\|\<z\>^su(t)\|_{L^{2}(\C)}  \sim \alpha^s_0K^{2s+1}|t|^{s}, \quad \|\<z\>^sv(t)\|_{L^{2}(\C)}  \sim \alpha^s_0K^{2s+1}|t|^{s}, \quad  t\longrightarrow \pm \infty,
\end{equation}
with $\dis \alpha_0= \frac{\sqrt{3}}{32 \pi}$.
 \end{cor}
 
 We have $u(t)= e^{-i\lambda t}R_{\alpha t}U $, then the previous bound can be directly obtained from 
  \begin{equation}\label{eq-stat} 
 \|\<z\>^s u(t)\|_{L^{2}(\C)} =\|\<z\>^s R_{\alpha t}U\|_{L^{2}(\C)}= \|\<z-\alpha t \>^s U\|_{L^{2}(\C)} \sim |\alpha|^s |t|^s  \|U\|_{L^{2}(\C)},
   \end{equation} 
when $t \longrightarrow \pm \infty$. Observe that by Theorem~\ref{borne-poly}, such a rate of growth  is excluded for \eqref{sys-sig} when $\sigma=1$, which gives another proof of the non-existence of  progressive waves in the case $\sigma=1$. The result of Corollary~\ref{cor-growth} shows that  growth of Sobolev norms for~\eqref{sys-sig} can occur even with small initial conditions. \medskip

There are some results on the growth of high Sobolev norms for  nonlinear Schr\"odinger equations, for large times. In \cite{HPTV} unbounded orbits were obtained  for NLS on a wave guide, thanks to a modified scattering result (see also  \cite{CKSTT, Hani,GuaKa} for norm inflation phenomena). In~\cite{Car-Galla} logarithmic growth of norms was obtained for logNLS. In~\cite{ANS}, the authors  prove exponential growth of the energy norm for  the harmonic oscillator perturbed by the angular momentum operator (this phenomenon also occurs for linear equations). Concerning  the Szeg\H{o} equation, growth of Sobolev norms was established in~\cite{P2} for the line and  in \cite{GG3} for the circle. See also \cite{Xu} for such results on  the half-wave equation on a wave guide. In \cite{Xu2, Thirouin, Biasi-Evnin}, examples of solutions with optimal exponential growth were obtained for  modifications of Szeg\H{o} equations. \medskip

It is also relevant to compare our results to the case of Schr\"odinger systems of the form
 \begin{equation}\label{sys-ref}
\left\{
\begin{aligned}
&i\partial_{t}u+\Delta u=  |v|^2 u, \quad   (t,x)\in \R\times \R^d \text{ or }    \R\times \T^d \text{ or }    \R\times (\R \times \T^d) ,\\
&i\partial_{t}v+\Delta v=  \sigma |u|^2 v,\\
&u(0,z)=  u_0(x),\; v(0,z)=  v_0(x).
\end{aligned}
\right.
\end{equation}
Most of the techniques developed for NLS can be adapted to study such systems. For global well-posedness results, we refer for instance to  \cite{MaSong}. It is likely that one can obtain polynomial bounds on the Sobolev norms   for \eqref{sys-ref} using the ideas of Sohinger~\cite{Sohinger1, Sohinger2}. A modified scattering result was obtained in \cite{Victor}, and the existence of unbounded orbits (on a wave guide) in the case $\sigma=1$ follows from \cite{HPTV}. In \cite{GPT} a non-linear phenomenon was exhibited on \eqref{sys-ref} posed on $\T$.  Moreover, let us mention that traveling waves solutions (in $H^1$) exist for Schr\"odinger equations on $\R^d$ and for coupled  Schr\"odinger systems as well : see in particular \cite{Ianni-LeCoz, DeleCozWeis} where traveling waves with different speeds are constructed. Notice that such solutions do not infer growth of Sobolev norms in this setting. There exist also traveling waves   nonzero conditions at infinity (see  \cite{CT, ChM} and references therein).  Finally, we refer to  reference~\cite{TzVi} for results on the growth of higher order moments of the linear and the non-linear Schr\"odinger equations with the same rate as in \eqref{rate}, but both   phenomena are different (in the present paper the growth is due to the nonlinearity, while in~\cite{TzVi} the mechanism  also holds for linear equations).\medskip

The results of Theorem~\ref{classi}  and Corollary~\ref{cor-growth} can be used to obtain growth of Sobolev norms of linear equations with time dependent potentials, see~\cite{Thomann}.\medskip

We conclude this paragraph by summing up a few general properties of  progressive waves solution to~\eqref{sys-sig}, of the form
 \begin{equation}\label{ansatz-2}
 \big(u(t,z), v(t,z)\big)=\big(e^{-i\lambda t}U(z+\alpha t) e^{\frac{1}{2}(\overline z \alpha - z \overline{\alpha})t},e^{-i \mu t}V(z+\alpha t) e^{\frac{1}{2}(\overline z \alpha - z \overline{\alpha})t} \big).
 \end{equation}
 
\begin{prop}\label{prop-gen}
Assume that $(u,v)\in L^{2,1}_\E \times L^{2,1}_\E$ is any progressive wave solution to \eqref{sys-sig} of the form \eqref{ansatz-2} with $\alpha \in \C^*$. Then 
\begin{enumerate}[$(i)$]
\item  {$M(U)=M(V)$} ;
\item  $\Re\big(\ov{\alpha} Q_-(U,V)\big)=0$ ; 
\item  $\Im\big(\ov{\alpha} Q_-(U,V)\big) \neq 0$ ; 
\item For $1\leq j\leq 4$, the couples $(u_j,v_j) \in L^{2,1}_\E \times L^{2,1}_\E$ are also  progressive wave solutions to~\eqref{sys-sig}, where 
\begin{align}
&(U_1, V_1)=\big(Ke^{i\theta_1}U, Ke^{i\theta_2}V\big), \quad {\alpha_1}= K^2\alpha  , \quad {\lambda_1}=K^2 \lambda,\quad  {\mu_1}=K^2 \mu\label{u0} \\
&(U_2, V_2)=L_{\theta}\big(U, V \big), \quad {\alpha_2}=\alpha e^{-i\theta}, \quad {\lambda_2}=\lambda,\quad  {\mu_2}=\mu
\label{u1}\\
&(U_3, V_3)=R_{\beta}\big(U,V \big), \quad {\alpha_3}=\alpha , \quad {\lambda_3}=\lambda+2 \Im(\beta \ov{\alpha} ),\quad  {\mu_3}=\mu+ 2 \Im(\beta \ov{\alpha})\label{u2}\\
&(U_4, V_4)=\big(V, U\big), \quad {\alpha_4}=-\alpha  , \quad {\lambda_4}=-\mu,\quad  {\mu_4}=-\lambda \label{u3}
\end{align}
for any $K\geq 0$, $\theta, \theta_1, \theta_2 \in \R$, and $\beta \in \C$\;;
 
\item The speed $\alpha \in \C$ is given by the formula
\begin{equation}\label{form-a}
\alpha =\frac{i}{M(U)}  \int_{\C}   |V|^2  {\partial_{\ov{z}}} (  |U|^2 )      dL:=F(U,V)\;;
\end{equation}
\item The following bound holds true: $\dis |\alpha| \leq \frac{ M(U)}{2\sqrt{2} \pi}$\;;
\item The parameters $\lambda, \mu \in \R$ and $\alpha \in \C$ are related by
$$(\lambda-\mu)M(U)-2i \ov{\alpha}Q_-(U,V)=2 \mathcal{H}(U,V).$$
\end{enumerate}
\end{prop}

 We  observe that the function $F$ in  \eqref{form-a}  satisfies  $F(R_\beta U,R_\beta V)=F(U,V)$, for any $\beta \in \C$, and $F(L_\theta U,L_\theta V)=e^{-i\theta }F(U,V)$, for any $\theta  \in \R$, which correspond to the symmetries of the problem.

As a consequence of point $(vi)$, the progressive waves found in Theorem~\ref{classi} have almost the maximal speed \big($|\alpha|\sim c_0 M(u)\big)$. Actually, we do not know whether there exist nontrivial progressive waves with an infinite number of zeros (such solution exist in the case $\alpha=0$, see \eqref{chi}). In particular, it would be interesting to see if there exist progressive waves with arbitrary small speed, at fixed $L^2$-norm \big($|\alpha|\ll  M(u)$\big). It is not clear how to use variational methods in this context, because of the lack of control of the~$L^{2,1}$ norm when $\sigma=-1$.  

\begin{rem}
  Assume that $U \in L^{2,1}_\E$ satisfies the equation
\begin{equation}\label{infini}
\lambda U+(\alpha \cdot \Gamma)U =\Pi\big(|U(-z)|^2U(z)\big),
\end{equation}
and denote by $V=L_\pi U$, namely $V(z)=U(-z)$, then $(U,V)$ is the initial condition of a progressive wave with speed $\alpha$ and $\mu=-\lambda$. Indeed by \eqref{r-2}, $V$ satisfies the equation 
\begin{equation*}
-\lambda V+(\alpha \cdot \Gamma)V =-\Pi\big(|U|^2V\big),
\end{equation*}
and by~\eqref{form-a}, the corresponding speed satisfies
\begin{equation*} 
\alpha =\frac{i}{M(U)}  \int_{\C}   |U(-z)|^2  {\partial_{\ov{z}}} \big(  |U(z)|^2 \big)      dL(z).
\end{equation*}
\end{rem}

\subsection{Plan of the paper} The rest of the article is organized as follows: in the last part of this section, we consider more general systems, we recall some harmonic analysis in Bargmann-Fock spaces and we precise some notations. Section~\ref{Sect2} is devoted to  the global existence results and we prove the polynomial bounds on  the solutions of the systems. Finally, in Section~\ref{Sect3} we prove the classification Theorem~\ref{classi} and Proposition~\ref{prop-gen}. We have added two  appendices, the first one contains some technical results and in the second one  we state improved  bounds on the solutions for the cubic LLL equation.

\begin{acknowledgements}
The second author warmly thanks Patrick Grard and Pierre Germain for many discussions on this subject. This work benefited from the previous collaboration~\cite{GGT}.
\end{acknowledgements}

\subsection{On related equations}

\subsubsection{Systems with more general nonlinearities}
It is likely that many of our results   can be adapted to the more general system
 \begin{equation*} 
\left\{
\begin{aligned}
&i\partial_{t}u= \alpha \Pi (|u|^2 u)+ \beta \Pi (|v|^2 u), \quad   (t,z)\in \R\times \C,\\
&i\partial_{t}v=  \gamma \Pi (|v|^2 v)+\sigma \beta \Pi (|u|^2 v),\\
&u(0,z)=  u_0(z),\; v(0,z)=  v_0(z),
\end{aligned}
\right.
\end{equation*}
where $\alpha, \beta, \gamma \in \R$ and $\sigma=\pm 1$. In this case the Hamiltonian reads
$$\mathcal{H}(u,v)=\frac{\alpha}2  \int_{\C}|u|^4 dL+\frac{\sigma \gamma}2  \int_{\C}|v|^4dL+\beta\int_{\C}|u|^2|v|^2dL.$$
The symmetries and the conservation laws of \eqref{sys-sig} also hold for this system. In particular, this excludes the existence of traveling waves in $L^{2,1}_\E$ in the case $\sigma=1$. We did not try to find progressive waves in the case $\sigma=-1$ for  the general system.

\subsubsection{Systems with dispersion}
 Denote by $H$ the harmonic oscillator which is defined by
$$
H = -4\partial_z \partial_{\ov z}+|z|^2.
$$  
This operator plays a key role in the study of Bargmann-Fock spaces and LLL. In particular, the following identity holds true 
\begin{equation}\label{conj}
e^{-itH}\Pi\big(e^{itH} a\, \ov{e^{itH} b}\, e^{itH} c\big)= \Pi\big( a\, \ov{b}\, c\big), \quad  \forall \,a,b,c  \in \E,
\end{equation}
see \cite[Lemma~2.4 and Corollary~2.5]{GHT1}. This identity relies on the particular resonant structure of the nonlinearity, and it is worth noticing  that there is no such relation for  Schr\"odinger equations with polynomial nonlinearities. Therefore, for a given dispersion parameter $\delta \in \R$,  the change of unknown $(\wt{u}, \wt{v}) =e^{-i \delta t H}(u,v)$, shows that  the system~\eqref{sys-sig} is equivalent to
 \begin{equation}\label{sys-conj}
\left\{
\begin{aligned}
&i\partial_{t}\wt{u}- \delta H\wt{u} =  \Pi (|\wt{v}|^2 \wt{u}), \quad   (t,z)\in \R\times \C,\\
&i\partial_{t}\wt{v}-\delta  H\wt{v}=  \sigma   \Pi (|\wt{u}|^2 \wt{v}),\\
&\wt{u}(0,z)=  u_0(z),\; \wt{v}(0,z)=  v_0(z).
\end{aligned}
\right.
\end{equation}
 Moreover, recall that $e^{i\tau H}$ is an isometry of the space
\begin{equation*} 
\HH^{s}(\C) = \big\{ u\in \mathscr{S}'(\C),\; {H}^{s/2}u\in L^2(\C)\big\},
\end{equation*}
and by testing on the complete  family $(\phi_n)_{n\geq 0}$, we can check that $e^{i \tau H}= e^{2i\tau }L_{2\tau}$. 
As a consequence, all results obtained for \eqref{sys-sig} also hold for \eqref{sys-conj}, and in the case $\sigma=-1$, we get the following progressive waves for \eqref{sys-conj}
 \begin{eqnarray*} 
 \big(\wt{u}, \wt{v}\big)&=&\big(e^{-i\lambda t}      e^{-i \delta t H} R_{\alpha t} U , e^{-i\mu t} e^{-i\delta t H} R_{\alpha t} V  \big)\\
 &=&\big(e^{-i(\lambda +2\delta)t}      L_{-2\delta t}R_{\alpha t} U , e^{-i(\mu +2\delta)t}      L_{-2\delta t} R_{\alpha t} V  \big),
 \end{eqnarray*}
 with $\lambda, \mu, \alpha$ and $(U,V)$ given by \eqref{prog}.
 
\subsubsection{A modified LLL equation} As for \eqref{sys-sig} we can show that the equation
 \begin{equation*} 
\left\{
\begin{aligned}
&i\partial_t u(t,z)=\Pi\Big(|u(t,-z)|^2u(t,z)\Big), \quad   (t,z)\in \R\times \C,\\
&u(0,z)=  u_0(z),
\end{aligned}
\right.
\end{equation*}
is globally well-posed in $\E$ as well as in the spaces $L^{2,s}_\E$, $s\geq 0$.  

The energy $\mathcal{H}(u)=\int_{\C}|u(-z)|^2|u(z)|^2dL(z)$ and the mass $M(u)=\int_{\C}|u(z)|^2dL(z)$ are conserved, moreover 
\begin{itemize}
\item[$\bullet$] we can prove the general bound $\| \<z\>^k u(t)\|_{L^2(\C)} \leq C \<t\>^k$ for any initial condition $u_0 \in L^{2,k}_\E$~;
\item[$\bullet$] the initial condition $\dis u_0=\frac12\phi_0+i\frac{\sqrt{3}}2\phi_1$ defines a progressive wave (see  \eqref{infini}). In particular, the previous bound is optimal.
\end{itemize}

\subsection{Analysis in the Bargmann-Fock space and notations} 
 
 We recall here some notations and basic facts which will be useful in the sequel. For more details on the analysis in the Bargmann-Fock space, we refer to the textbook  \cite{Zhu}. \medskip
 
 The harmonic oscillator $H$ is defined by
$$
H = -4\partial_z \partial_{\ov z}+|z|^2.
$$
Denote by $(\phi_n)_{n \geq 0}$ the family of the special Hermite functions given by 
$$
\varphi_n(z) = \frac{1}{\sqrt{\pi n!}} z^n e^{-\frac{|z|^2}{2}}.
$$
The family $(\phi_n)_{n \geq 0}$ forms  a Hilbertian basis of $\mathcal{E}$ (see \cite[Proposition 2.1]{Zhu}),  and the $\phi_n$ are the eigenfunctions of $H$, namely 
$$H\phi_n=2(n+1)\phi_n, \quad n\geq 0.$$
For $\gamma \in \C$, we define 
\begin{equation}\label{defphia}
\phi_n^\gamma(z) = R_{- \overline \gamma} (\phi_n)(z) = \frac{1}{\sqrt{\pi n!}} (z-\overline \gamma)^n e^{-\frac{|z|^2}{2}-\frac{|\gamma|^2}{2} + \gamma z}.
\end{equation} 

The   kernel of  $\Pi$, the orthogonal projection on $\mathcal{E}$, is explicitly given by
\begin{equation*}
K(z,\xi)=\sum_{n=0}^{+\infty}\phi_n(z)\ov{\phi_n}(\xi)=\frac{1}{\pi}e^{\ov{\xi}z}e^{-\vert \xi\vert^2/2}e^{-\vert z\vert^2/2}, \quad   (z,\xi)\in \C\times \C,
\end{equation*} 
and therefore we get the formula 
\begin{equation}\label{defpi}
[\Pi u](z) = \frac{1}{\pi} e^{-\frac{|z|^2}{2}} \int_\mathbb{C} e^{\ov  w z - \frac{|w|^2}{2}} u(w) \,dL(w),
\end{equation}
where $L$ stands for Lebesgue measure on $\C$. With this  formula, we can compute   the following product rule (see \cite[Lemma 8.1]{GHT1})
   \begin{equation}  \label{pi-phi}
\Pi\big(\phi_{n_1}\ov{\phi_{n_2}}\phi_{n_3}\big) = \left\{
\begin{aligned}
& \frac{1 }{2\pi} \frac{(n_1+n_3)!}{2^{n_1+n_3}\sqrt{n_1 !n_2 !n_3 !n_4 !}} \phi_{n_4} &\quad \text{if} \quad n_4:= n_1+n_3-n_2\geq 0 \;\\
& 0&\quad \text{if} \quad n_4:= n_1+n_3-n_2 <0 .
\end{aligned}
\right.
\end{equation}

\medskip
Throughout the paper we use the classical notations $z=x+iy$ and 
$$\partial_z=\frac12(\partial_x-i\partial_y), \qquad \partial_{\ov{z}}=\frac12(\partial_x+i\partial_y).$$ 
We define the enlarged lowest Landau level space as
\begin{equation*} 
\widetilde{\mathcal E}=\Big\{ u(z) = e^{-\frac{|z|^2}{2}} f(z)\,,\;f \; \mbox{entire\ holomorphic}\, \Big\}\cap \mathscr{S}'({\mathbb C})=\Big\{ u\in \mathscr{S}'(\C ), {\pa_{\ov z}}u+\frac	{z}{2}u=0\Big\} \ .
\end{equation*}
It is remarkable that in $\widetilde{\mathcal{E}}$ we have embeddings of $L^p$ spaces. Namely, by Carlen~\cite{Carlen}, for all~$u \in \widetilde{\mathcal{E}}$  the following hypercontractivity estimates hold true
\begin{equation}  \label{hypercontract}
\mbox{if \;$1 \leq p \leq q \leq +\infty$,} \qquad\left( \frac{q}{2\pi} \right)^{1/q} \| u \|_{L^q(\C)} \leq \left( \frac{p}{2\pi} \right)^{1/p} \| u \|_{L^p(\C)}.
\end{equation} 
We refer to Lemma~\ref{lem.hyp} for an elementary proof but without the optimal constants.\medskip

For $s\in \R$, we denote by 
$$L^{2,s}=\big\{u\in \mathscr{S}'(\C), \;\<z\>^su \in L^2(\C)\big\}, \quad \<z\>=(1+|z|^2)^{1/2}$$
 the weighted Lebesgue space and 
 $$L^{2,s}_{\E}=L^{2,s}  \cap \wt{\E}.$$
   For $s \in \mathbb{R}$,  we define the harmonic Sobolev spaces    by 
\begin{equation} \label{def-sobo}
\HH^{s}(\C) = \big\{ u\in \mathscr{S}'(\C),\; {H}^{s/2}u\in L^2(\C)\big\}\cap \wt{\E},
\end{equation}
equipped with the natural norm $ \|u\|_{\HH^s(\C)} =\|H^{s/2} u\|_{L^2(\C)}$. Then, we have $\HH^{s}(\C)  = L^{2,s}_{\E}$ and the following equivalence of norms holds true
\begin{equation*} 
c\|\<z\>^s u\|_{L^2(\C)} \leq  \|u\|_{\HH^s(\C)} \leq C\|\<z\>^s u\|_{L^2(\C)}, \quad \forall\,u \in L^{2,s}_{\E},
\end{equation*}
see~\cite[Lemma~C.1]{GGT} for a proof.  \medskip

In this paper $c,C>0$ denote universal constants the value of which may change from line to line.


 \section{Global existence results and bounds on Sobolev norms}\label{Sect2}

  \subsection{Global existence: proof of Theorem \ref{mainCauchy}}  
  
 The   well-posedness arguments are elementary, therefore we only give the main lines.\medskip
   
$\bullet$  Denote by $U=(u,v)$ and  $f(U)=\big(\Pi(|v|^2u), \sigma\Pi(|u|^2v)\big)$. Then equation~\eqref{sys-sig} is equivalent to 
  $$F(U)=U_0-i\int_{0}^tf(U)(s)ds.$$
  We will  solve this integral equation   thanks to a fixed point argument: define the norms
$$\dis \|U\|_T=\sup_{t\in [0,T]}\big(\|u(t)\|_{L^2(\C)} +\|v(t)\|_{L^2(\C)}  \big) \quad \text{and}\quad  \|U_0\|= \|u_0\|_{L^2(\C)} +\|v_0\|_{L^2(\C)},$$
  then, using the Carlen estimate \eqref{hypercontract}, we get
  \begin{equation*}
  \|F(U)\|_T \leq \|U_0\| +T   \|f(U)\|_T  \leq \|U_0\| +C T   \|U\|^3_T.
  \end{equation*}
  Similarly we obtain the contraction estimates 
    \begin{equation*}
    \|F(U_1)-F(U_2)\|_T \leq C T   \|U_1-U_2\|_T\big(\|U_1\|^2_T+ \|U_2\|^2_T\big).
        \end{equation*}
As a consequence, these estimates allow to perform a fixed point argument in the space
     $$X_{T}= \big\{U\in \mathcal{C}([0,T]; \E \times \E),\;\;   \|U\|_T \leq 2\|U_0\| \big\},$$
 when  $T\leq c \|U_0\|^{-2}$ for some small absolute constant $c>0$.\medskip
     
 $\bullet$     Global existence is obtained thanks to the conservation of the mass: $\|U(T)\|=\|U_0\|$ (since the local time of existence only depends on $\|U_0\|$).\medskip

$\bullet$   Assume that   for some $s>0$, $(\< z\> ^s u_0, \<z\rangle ^s v_0)\in L^2(\C)  \times  L^2(\C )$.  Then we can adapt the previous fixed point argument in weighted spaces, using the estimate 
$$
\| \langle z \rangle^s \Pi \big( a b c\big) \|_{L^2}  \leq C \|\langle z \rangle^s a\|_{L^2} \| b\|_{L^2}\| c\|_{L^2},
$$
(see~\cite[Propositions 3.1 and 3.2]{GGT}). The time of existence obtained in the argument only depends on the $L^2$ norm, hence the solution can be globalized as in the case $s=0$.

  \subsection{Non dispersion: proof of Proposition \ref{propLinfini}}  
The upper bound is a consequence of~\eqref{hypercontract}. Let us prove to lower bound: by the conservation of $\mathcal{H}$ under the flow and the H\"older inequality, 
\begin{equation*}
\mathcal{H}(u_0,v_0)=\mathcal{H}(u,v) \leq \|u\|^2_{L^4} \|v\|^2_{L^4} \leq \|u\|_{L^2}\|u\|_{L^\infty}\|v\|_{L^2}\|v\|_{L^\infty}.
\end{equation*}
Besides, by~\eqref{hypercontract} again, we have $ \|v\|_{L^\infty} \leq \frac{1}{\sqrt{\pi}}\|v\|_{L^2}$, and gathering the previous inequalities we get the result.

  \subsection{Bounds on Sobolev norms}

Let us now turn to the proofs of  Theorem~\ref{borne-poly} and Theorem~\ref{thm1.6}.  The proofs of both results follow the same strategy, but in the case $\sigma=1$ one can take advantage of  the conservation of the $L^{2,1}$-norm and of additional cancellations to improve the bounds. 

\subsubsection{A technical result}
We first state a result which gives a precise description of the derivative of the $L^{2,k}$-norm. Here the notation $W \in \mathcal{C}_c^k(\R \times \R^2, \R)$ means continuity in $t$ of the   derivatives in $(x,y)$ of order less that $k \in \N$.

\begin{lem}\label{lem3.1}
Let $k \in \N$ and let $W \in \mathcal{C}_c^k(\R \times \R^2, \R)$ be a real valued function. Assume that $u\in L^{2,k}_\mathcal{E}$ satisfies
$$i \partial_t u=\Pi \big(W u\big).$$
Then 
\begin{equation}\label{form-deri}
\frac{d}{dt}\int_{\C}|z|^{2k}|u(t,z)|^2dL(z)= -2\sum_{j=1}^k(-1)^j{{k}\choose{j}}\mathfrak{Im} \int_{\C}    z^k    \ov{z}^{k-j}  |u(t,z)|^2 \big( \partial_z^j     W(t,z)\big)dL(z).
\end{equation}
\end{lem}

\begin{proof} Let us first show that for $u \in \mathscr{S}'(\C)$, 
     \begin{equation}\label{piu} 
\Pi\big(\ov{z}u\big)=\big(\partial_z+\frac{\ov{z}}{2}\big)\Pi u.
\end{equation}
We compute
\begin{eqnarray*}
\partial_z  \Pi u(z) &=&-  \frac{1}{\pi}\frac{\ov{z}}2 e^{-\frac{|z|^2}{2}} \int_\mathbb{C} e^{\ov  w z - \frac{|w|^2}{2}} u(w) \,dL(w)+ \frac{1}{\pi}  e^{-\frac{|z|^2}{2}} \int_\mathbb{C} e^{\ov  w z - \frac{|w|^2}{2}} \ov{w}u(w) \,dL(w)\\
&=&-\frac{\ov{z}}{2}\Pi u(z)+\Pi\big(\ov{z}u\big)(z),
\end{eqnarray*}
hence the result. Next, by commutation of the operators $\partial_z$ and $\ov{z}$ we get, for any $u\in\wt{\E}$ and $k\in \N$
  \begin{equation} \label{zk}
\Pi\big(|z|^{2k}u\big)=\big(\partial_z+\frac{\ov{z}}{2}\big)^k(z^ku).
\end{equation}
Now, we compute
\begin{eqnarray*}
\frac{d}{dt}\int_{\C}|z|^{2k}|u|^2dL(z)&=& 2\mathfrak{Re} \int_{\C} |z|^{2k} \ov{u}\partial_t u dL(z)\\
&=& 2\mathfrak{Im} \int_{\C} |z|^{2k} \ov{u}\Pi(Wu) dL(z)\\
&=& -2\mathfrak{Im} \int_{\C} \Pi(|z|^{2k}u) W\ov{u} dL(z).
\end{eqnarray*}
Then  by \eqref{zk},  and writing  $u(z)=f(z)e^{-\frac{|z|^2}{2}}$,
\begin{eqnarray*}
\frac{d}{dt}\int_{\C}|z|^{2k}|u|^2dL(z)&=&-2\mathfrak{Im} \int_{\C}\Big[\big(\partial_z+\frac{\ov{z}}{2}\big)^k(z^ku)\Big]W\ov{u} dL(z)\\
&=& -2\mathfrak{Im} \int_{\C}z^ku\big(-\partial_z+\frac{\ov{z}}{2}\big)^k(W\ov{u}) dL(z)\\
&=& -2\mathfrak{Im} \int_{\C}z^k|f|^2e^{-\frac{|z|^2}{2}}\big(-\partial_z+\frac{\ov{z}}{2}\big)^k(W e^{-\frac{|z|^2}{2}}) dL(z).
\end{eqnarray*}
For all $k\geq 0$, we have
$$\big(-\partial_z+\frac{\ov{z}}{2}\big)^k(We^{-\frac{|z|^2}{2}} )=e^{-\frac{|z|^2}{2}}   \big(-\partial_z+\ov{z}\big)^k W=e^{-\frac{|z|^2}{2}}  \sum_{j=0}^k  {{k}\choose{j}} \ov{z}^{k-j} (-\partial_z)^j W, $$
where the binomial formula    can be used since  $\partial_z$ and  $\ov{z}$ commute. Finally, from the previous line we deduce
\begin{equation*}
\frac{d}{dt}\int_{\C}|z|^{2k}|u(t,z)|^2dL(z)= -2\sum_{j=1}^k(-1)^j  {{k}\choose{j}} \mathfrak{Im} \int_{\C}   z^k    \ov{z}^{k-j}  |u(t,z)|^2 \big( \partial_z^j     W(t,z)\big)dL(z),
\end{equation*}
(the contribution $j=0$ vanishes since $W$ is real), which was the claim.
\end{proof}

\subsubsection{Proof of Theorem~\ref{thm1.6}}
We apply Lemma~\ref{lem3.1} with $W=|v|^2$, hence
\begin{eqnarray*}
\frac{d}{dt}\int_{\C}|z|^{2k}|u|^2dL(z)&=&   -2\sum_{j=1}^k(-1)^j {{k}\choose{j}}\mathfrak{Im} \int_{\C}   z^k    \ov{z}^{k-j}  |u|^2 \partial_z^j     (|v|^2)dL(z)   \\
 &\leq &C \max_{1\leq j \leq k}\big\|\partial_z^j     (|v|^2)\big\|_{L^\infty(\C)} \int_{\C}  \<z\>^{2k-1} |u|^2dL(z) \\
 &\leq &C \|v\|^2_{L^\infty(\C)} \int_{\C}  \<z\>^{2k-1} |u|^2dL(z),
\end{eqnarray*}
where in the last line we used Lemma~\ref{lem.deri}. Finally, thanks to the H\"older inequality we deduce that 
\begin{eqnarray*} 
\frac{d}{dt} \big\| \<z\>^ku\big\|^2_{L^2(\C)} &\leq &C \|v\|^2_{L^2(\C)}  \|u\|^{\frac 1k}_{L^2(\C)}    \big\| \<z\>^ku\big\|^{2-\frac1{k}}_{L^2(\C)} \\
&\leq & C \|v_0\|^2_{L^2(\C)}   \|u_0\|^{\frac 1k}_{L^2(\C)} \big\| \<z\>^ku\big\|^{2-\frac1{k}}_{L^2(\C)},
\end{eqnarray*}
which in turn implies 
\begin{eqnarray*} 
 \big\| \<z\>^ku(t)\big\|_{L^2(\C)} &\leq &\Big( \big\| \<z\>^ku_0\big\|^{\frac1k}_{L^2(\C)}+C \|v_0\|^2_{L^2(\C)}   \|u_0\|^{\frac 1k}_{L^2(\C)} |t| \Big)^k\\
 &\leq &  \big\| \<z\>^ku_0\big\|_{L^2(\C)} \big(1+C \|v_0\|^2_{L^2(\C)}  |t| \big)^k, \nonumber
\end{eqnarray*}
hence the result.  Similarly, we get  
$$ \big\| \<z\>^kv(t)\big\|_{L^2(\C)}  \leq   \big\| \<z\>^kv_0\big\|_{L^2(\C)} \big(1+C \|u_0\|^2_{L^2(\C)}  |t| \big)^k,$$ which completes the proof. \medskip

The proof of \eqref{borneck} is postponed to Paragraph~\ref{para-2.2.4}.

\subsubsection{Proof of Theorem~\ref{borne-poly}}

We apply \eqref{form-deri} with $W=|v|^2$ for the first part and  $W=|u|^2$ for the second, then 
\begin{equation}\label{derk2}
\frac{d}{dt}\int_{\C}|z|^{2k}\big(|u|^2+|v|^2\big)dL=   -2\sum_{j=1}^k(-1)^j {{k}\choose{j}}\mathfrak{Im} \int_{\C}   z^k    \ov{z}^{k-j}  \big(|v|^2 \partial_z^j     (|u|^2) + |u|^2 \partial_z^j     (|v|^2)\big)dL.
\end{equation} 

$(i)$ To begin with, we consider the case $k\geq 4$.\medskip

$\bullet$ Let us look more carefully at the contribution $j=1$. We have 
\begin{eqnarray}
\mathfrak{Im} \int_{\C}   z^k    \ov{z}^{k-1}  \big(|v|^2 \partial_z     (|u|^2)+|u|^2 \partial_z     (|v|^2)\big)dL  &= & \mathfrak{Im} \int_{\C}   z^{k}    \ov{z}^{k-1}   \partial_z     (|u|^2|v|^2)dL \nonumber\\
 &= &- k\mathfrak{Im} \int_{\C}   |z|^{2(k-1)}       |u|^2|v|^2dL\nonumber\\
 &=&0.\label{ipp2}
 \end{eqnarray}

$\bullet$ Let us consider  the contribution $j=2$. We have 
\begin{equation}\label{ordre2b}
|z|^2 \partial^2_z(|u|^2)=\partial^2_{z}\big(|zu|^2\big)-2\ov{z}\partial_z(|u|^2),
\end{equation}
thus the contribution for $j=2$ reads
\begin{multline*}
 \mathfrak{Im} \int_{\C}   z^k \ov{z}^{k-2} \Big(    |v|^2 \partial^2_z     (|u|^2)+  |u|^2 \partial^2_z     (|v|^2)  \Big) dL=\\
 =\mathfrak{Im} \int_{\C}   z^{k-1}    \ov{z}^{k-3}  \Big(  |v|^2 \partial^2_{z}\big(|zu|^2\big)+  |u|^2 \partial^2_{z}\big(|zv|^2\big)\Big)dL\\-2  \mathfrak{Im} \int_{\C}   z^{k-1}    \ov{z}^{k-2}   \Big( |v|^2 \partial_z(|u|^2)+ |u|^2 \partial_z(|v|^2)\Big)dL.
 \end{multline*}
 The first part can be controlled by
\begin{eqnarray*}
\big|\mathfrak{Im} \int_{\C}   z^{k-1}    \ov{z}^{k-3}  |v|^2 \partial^2_{z}\big(|zu|^2\big) dL \big| &\leq  &\big\| \partial^2_z\big(|zu|^2\big)\big\|_{L^\infty(\C)}    \int_{\C}   \<z\>^{2k-4}    |v|^2     dL   \\
   &\leq & C   \Vert \langle z\rangle u_0\Vert^2_{L^2(\C)}   \Big(\int_{\C}  \<z\>^{2k} |v|^2dL\Big)^{1-\frac{2}{k-1}} \Big(\int_{\C}  \<z\>^{2} |v|^2dL\Big)^{\frac{2}{k-1}} \\
      &\leq & C   \Vert \langle z\rangle u_0\Vert^{2}_{L^2(\C)}\Vert \langle z\rangle v_0\Vert^{\frac4{k-1}}_{L^2(\C)}   \Big(\int_{\C}  \<z\>^{2k} |v|^2dL\Big)^{1-\frac{2}{k-1}}, 
 \end{eqnarray*}
and same for the analogous term:
\begin{equation*}
\big|\mathfrak{Im} \int_{\C}   z^{k-1}    \ov{z}^{k-3}  |u|^2 \partial^2_{z}\big(|zv|^2\big) dL \big| \leq  C   \Vert \langle z\rangle v_0\Vert^{2}_{L^2(\C)}\Vert \langle z\rangle u_0\Vert^{\frac4{k-1}}_{L^2(\C)}   \Big(\int_{\C}  \<z\>^{2k} |u|^2dL\Big)^{1-\frac{2}{k-1}} .
 \end{equation*}
 The remaining part vanishes, as in \eqref{ipp2}
\begin{equation*}
 \mathfrak{Im} \int_{\C}   z^{k-1}    \ov{z}^{k-2}  \Big(|v|^2  \partial_z(|u|^2)+|u|^2  \partial_z(|v|^2)\Big)dL  =0.
 \end{equation*}

$\bullet$ For the contributions $j=3$  we write (using \eqref{ordre2b})
$$|z|^2 \partial^3_z(|u|^2)=\partial^3_{z}\big(|zu|^2\big)-3\ov{z}\partial^2_z(|u|^2),$$
hence, by Lemma~\ref{lem.deri}
\begin{multline*}
\big|\mathfrak{Im} \int_{\C}   z^{k}    \ov{z}^{k-3}  |v|^2 \partial^3_{z}\big(|u|^2\big) dL \big| \leq  \\
\begin{aligned}
&\leq \big\| \partial^3_z\big(|zu|^2\big)\big\|_{L^\infty(\C)}    \int_{\C}   \<z\>^{2k-5}    |v|^2     dL  +3 \big\| \partial^2_z\big(|u|^2\big)\big\|_{L^\infty(\C)}    \int_{\C}   \<z\>^{2k-4}    |v|^2     dL  \\
      &\leq C    \Vert \langle z\rangle u_0\Vert^{2}_{L^2(\C)}\Vert \langle z\rangle v_0\Vert^{\frac4{k-1}}_{L^2(\C)}    \Big(\int_{\C}  \<z\>^{2k} |v|^2dL\Big)^{1-\frac{2}{k-1}} .
 \end{aligned}
 \end{multline*}~

$\bullet$ The contributions $4 \leq j\leq k$ can directly be controlled.\medskip

As a consequence, if we set $X_k(t)=\Vert \langle z\rangle ^ku(t)\Vert _{L^2(\C)}+\Vert \langle z\rangle ^kv(t)\Vert _{L^2(\C)}$
\begin{eqnarray*} 
\frac{d}{dt} \Big(\big\| \<z\>^ku\big\|^2_{L^2(\C)}+\big\| \<z\>^kv\big\|^2_{L^2(\C)} \Big) & \leq & CX^{2+\frac4{k-1}}_1(0)   \Big(\big\| \<z\>^ku\big\|^{2-\frac{4}{k-1}} _{L^2(\C)}+\big\| \<z\>^kv\big\|^{2-\frac{4}{k-1}} _{L^2(\C)}\Big) \\
& \leq & CX^{2+\frac4{k-1}}_1(0)  \Big(\big\| \<z\>^ku\big\| ^2_{L^2(\C)}+\big\| \<z\>^kv\big\|^2_{L^2(\C)}\Big)^{1-\frac{2}{k-1}}  ,
\end{eqnarray*}
which in turn implies after integration
\begin{equation*} 
 \big\| \<z\>^ku(t)\big\|^2_{L^2(\C)}+ \big\| \<z\>^kv(t)\big\|^2_{L^2(\C)}  \leq C X^2_k(0) \big(1+X^2_1(0) |t| \big)^{\frac{k-1}2},
\end{equation*}
and then 
\begin{equation}\label{borne-expl}
 \big\| \<z\>^ku(t)\big\|_{L^2(\C)}+ \big\| \<z\>^kv(t)\big\|_{L^2(\C)}  \leq C X_k(0) \big(1+X^2_1(0) |t| \big)^{\frac{k-1}4}.
\end{equation}~

$(ii)$ Now we consider the case $k=2$. By \eqref{derk2} and~\eqref{ipp2} we have 
\begin{equation*}
\frac{d}{dt}\int_{\C}|z|^{4}\big(|u|^2+|v|^2\big)dL=   -2 \mathfrak{Im} \int_{\C}   z^2     \big(  |v|^2 \partial_z^2 (|u|^2)+ |u|^2 \partial_z^2 (|v|^2)\big)dL,
\end{equation*} 
thus 
\begin{eqnarray*}
\frac{d}{dt}\int_{\C}|z|^{4}\big(|u|^2+|v|^2 \big)dL & \leq &  2  \big( \big\| \partial^2_z\big(|u|^2\big)\big\|_{L^\infty(\C)}   \Vert \langle z\rangle v\Vert^2_{L^2(\C)} + \big\| \partial^2_z\big(|v|^2\big)\big\|_{L^\infty(\C)}   \Vert \langle z\rangle u\Vert^2_{L^2(\C)}\big)\\
& \leq & C \big\| \<z\> u_0\big\|^2_{L^2(\C)}\big\| \<z\> v_0\big\|^2_{L^2(\C)}  ,
\end{eqnarray*} 
which implies the announced bound. \medskip

$(iii)$  Finally we consider the case  $k=3$. By \eqref{derk2} and \eqref{ipp2} we have 
\begin{multline*}
\frac{d}{dt}\int_{\C}|z|^{6}\big(|u|^2+|v|^2\big)dL =  \\
=  -6 \mathfrak{Im} \int_{\C}   z^3 \ov{z}     \big( |v|^2 \partial_z^2     (|u|^2)+|u|^2 \partial_z^2     (|v|^2)\big)dL+
 2 \mathfrak{Im} \int_{\C}   z^3    \big(  |v|^2 \partial_z^3     (|u|^2)+ |u|^2 \partial_z^3     (|v|^2)\big)dL.
\end{multline*} 
The first term can be controlled as in the case $k\geq 4$. For the second one, we write 
\begin{multline*}
\big|\mathfrak{Im} \int_{\C}   z^3    \big(  |v|^2 \partial_z^3     (|u|^2)+ |u|^2 \partial_z^3     (|v|^2)\big)dL \big| \leq  \\
\begin{aligned}
& \leq \big\| z^2 \partial^3_z\big(|u|^2\big)\big\|_{L^\infty(\C)} \Vert \langle z\rangle v\Vert_{L^2(\C)}  \Vert   v\Vert_{L^2(\C)} +  \big\| z^2 \partial^3_z\big(|v|^2\big)\big\|_{L^\infty(\C)} \Vert \langle z\rangle u\Vert_{L^2(\C)}  \Vert   u\Vert_{L^2(\C)}\\
& \leq C\Vert \langle z\rangle u\Vert^2_{L^2(\C)}\Vert \langle z\rangle v\Vert^2_{L^2(\C)}
 \end{aligned}
 \end{multline*}~
 by Lemma~\ref{lem.deri} and the hypercontractivity estimates~\eqref{hypercontract}.
\medskip

$(iv)$  Estimate \eqref{newB2} is obtained by interpolation between the cases $k=1$ and $k=3$.

\subsubsection{Proof of the bound \eqref{borneck}}\label{para-2.2.4}
We proceed as in the proof of Theorem \ref{borne-poly}. First, we apply \eqref{form-deri} with $W=|v|^2$ for $v$, then
\begin{equation*}
\frac{d}{dt}\int_{\C}|z|^{2k}|u|^2dL=   -2\sum_{j=1}^k(-1)^j {{k}\choose{j}}\mathfrak{Im} \int_{\C}   z^k    \ov{z}^{k-j}  |u|^2 \partial_z^j     (|v|^2)dL.
\end{equation*} 
Similarly, we use \eqref{form-deri} with $W=-|u|^2$ and get
\begin{equation*}
\frac{d}{dt}\int_{\C}|z|^{2k}|v|^2dL=   -2\sum_{j=1}^k(-1)^j {{k}\choose{j}}\mathfrak{Im} \int_{\C}   z^k    \ov{z}^{k-j}  |v|^2 \partial_z^j     (-|u|^2)dL,
\end{equation*} 
so that
\begin{equation}\label{derk3}
\frac{d}{dt}\int_{\C}|z|^{2k}\big(|u|^2-|v|^2\big)dL=   -2\sum_{j=1}^k(-1)^j {{k}\choose{j}}\mathfrak{Im} \int_{\C}   z^k    \ov{z}^{k-j}  \big(|v|^2 \partial_z^j     (|u|^2) + |u|^2 \partial_z^j     (|v|^2)\big)dL.
\end{equation} 
$(i)$ To begin with, we consider the case $k\geq 3$.\medskip

$\bullet$ Let us look at the contribution $j=1$. We have just as \eqref{ipp2}
\begin{equation}\label{ipp3}
\mathfrak{Im} \int_{\C}   z^k    \ov{z}^{k-1}  \big(|v|^2 \partial_z     (|u|^2)+|u|^2 \partial_z     (|v|^2)\big)dL  = 0.
 \end{equation}

$\bullet$ For the contributions $2 \leq j \leq k$  we write by Lemma~\ref{lem.deri} and Theorem \ref{thm1.6} :
\begin{multline*}
\big|\mathfrak{Im} \int_{\C}   z^k    \ov{z}^{k-j}  \big(|v|^2 \partial_z^j     (|u|^2) + |u|^2 \partial_z^j     (|v|^2)\big)dL \big| \leq  \\
\begin{aligned}
&\leq \big\| \partial^j_z\big(|u|^2\big)\big\|_{L^\infty(\C)}    \int_{\C}   \<z\>^{2k-j}    |v|^2     dL  + \big\| \partial^j_z\big(|v|^2\big)\big\|_{L^\infty(\C)}    \int_{\C}   \<z\>^{2k-j}    |u|^2     dL  \\
      &\lesssim   \big(1+|t|\big)^{2k-2}.
 \end{aligned}
 \end{multline*}
 
By the binomial formula, we write 
\begin{equation*}
\big\| \<z\>^ku\big\|^2_{L^2(\C)}-\big\| \<z\>^kv\big\|^2_{L^2(\C)}= \sum_{j=0}^k{{k}\choose{j}}\int_{\C}|z|^{2j}\big(|u|^2-|v|^2\big)dL.
\end{equation*}
As a consequence,
\begin{equation*} 
\frac{d}{dt} \Big(\big\| \<z\>^ku\big\|^2_{L^2(\C)}-\big\| \<z\>^kv\big\|^2_{L^2(\C)} \Big) \lesssim  \big(1+|t|\big)^{2k-2},
\end{equation*}
which in turn implies after integration
\begin{equation*} 
 \Big| \big\| \<z\>^ku\big\|^2_{L^2(\C)}-\big\| \<z\>^kv\big\|^2_{L^2(\C)} \Big| \lesssim  \big(1+|t|\big)^{2k-1},
\end{equation*}
as announced.

$(ii)$ Now we consider the case $k=2$. By \eqref{derk3} and \eqref{ipp3} we have 
\begin{equation*}
\frac{d}{dt}\int_{\C}|z|^{4}\big(|u|^2-|v|^2\big)dL=   -2 \mathfrak{Im} \int_{\C}   z^2     \big(  |v|^2 \partial_z^2 (|u|^2)+ |u|^2 \partial_z^2 (|v|^2)\big)dL,
\end{equation*} 
thus 
\begin{eqnarray*}
\frac{d}{dt}\int_{\C}|z|^{4}\big(|u|^2-|v|^2 \big)dL & \lesssim &  \big\| \partial^2_z\big(|u|^2\big)\big\|_{L^\infty(\C)}   \Vert \langle z\rangle v\Vert^2_{L^2(\C)} + \big\| \partial^2_z\big(|v|^2\big)\big\|_{L^\infty(\C)}   \Vert \langle z\rangle u\Vert^2_{L^2(\C)}\\
& \lesssim & \big(1+|t|\big)^{2},
\end{eqnarray*} 
by Theorem \ref{thm1.6}, which implies the announced bound.


\section{Progressive waves with a finite number of zeros}\label{Sect3}

\subsection{The classification result: proof of Theorem \ref{classi}} Assume that $(U,V)$ satisfies the system~\eqref{syst}. We  adopt here the same strategy as in \cite[Section 6]{GGT} (see also \cite{AftaSerfa} for a similar approach in the context of the LLL equation on lattices).
\medskip 

 {\underline    {Step 1 : the Ansatz.}}  We write $U(z) = f(z)e^{-\frac{1}{2}|z|^2}$ and $V(z) =g(z) e^{-\frac{1}{2}|z|^2}$,  where $f$ and $g$ are entire functions which have a finite number of zeros. By  the argument of \cite[Section 6.3, Step~1]{GGT},   $U$ and $V$ take the form
 $$U(z)=P_1(z)e^{A_1z^2+B_1 z-\frac12|z|^2}, \quad V(z)=P_2(z)e^{A_2z^2+B_2 z-\frac12|z|^2},$$
  where $P_1$ and $P_2$ are polynomials, $deg P_1=m_1$, $deg P_2=m_2$ and $A_1, A_2, B_1, B_2$ are complex numbers. Moreover, since $U,V \in L^2(\C)$, we deduce that  ${|A_1|<\frac12}$ and ${|A_2|<\frac12}$.  \medskip
  
  {\underline    {Step 2 :  $A_1=A_2=0$.}} In the sequel we will use  the formula
\begin{equation*}
\frac{1}{\pi} \int e^{-2 |w|^2 + aw + b\overline{w} + cw^2 + d\overline{w^2}}\,dL(w) = \frac{1}{2 \sqrt{1-cd}} e^{\frac{da^2+cb^2+2ab}{4(1-cd)}},
\end{equation*}
which holds true   for any complex numbers $a,b,c,d$ such that the integral converges absolutely (the proof is elementary by computing Gaussian integrals). This identity implies, for a polynomial $P$ of $w$ and $\overline{w}$
\begin{equation}\label{inte-poly}
\frac{1}{\pi} \int e^{-2 |w|^2 + aw + b\overline{w} + cw^2 + d\overline{w^2}}P(w,\overline{w})\,dL(w) =  \frac{1}{2 \sqrt{1-cd}} P(\partial_a,\partial_b)e^{\frac{da^2+cb^2+2ab}{4(1-cd)}}.
\end{equation}
Therefore (see \cite[equality (6.9)]{GGT} for a similar formula),
\begin{align*}
\Pi (|V|^2 U) (z) & = \frac{e^{-\frac{|z|^2}{2}}}{\pi} \int e^{-2|w|^2 + z\overline{w} + (A_1+A_2) w^2 + \overline{A_2w^2} + (B_1+B_2)w + \overline{B_2 w}} P_1(w)  P_2(w) \ \overline{P_2(w)}\,dL(w) \\
& = \left.  \frac{e^{-\frac{|z|^2}{2}}}{2 \sqrt{1-cd}}  P_1(\partial_a)P_2(\partial_a) \overline{P_2}(\partial_b) e^{\frac{da^2+cb^2+2ab}{4(1-cd)}} \right|_{\substack{a = B_1+B_2 \\ b = z+\overline{B_2} \\ c = A_1+A_2 \\ d = \overline{A_2}}}. 
\end{align*}
Now, let  $(U,V)$ satisfy the system \eqref{syst}. We identify the coefficient of $z^2$ in the exponential and we get 
    \begin{equation} \label{4.1}
\left\{
\begin{aligned}
&\frac{A_1+A_2}{4\big(1-(A_1+A_2)\ov{A_2}\big)}=A_1\\
&\frac{A_1+A_2}{4\big(1-(A_1+A_2)\ov{A_1}\big)}=A_2.
\end{aligned}
\right.
\end{equation}
 We multiply the first line by $\ov{A_2}$ and the second by $\ov{A_1}$ and therefore 
     \begin{equation}  \label{4.2}
\left\{
\begin{aligned}
&\frac{(A_1+A_2)\ov{A_2}}{4\big(1-(A_1+A_2)\ov{A_2}\big)}=A_1 \ov{A_2}\\
&\frac{(A_1+A_2)\ov{A_1}}{4\big(1-(A_1+A_2)\ov{A_1}\big)}=A_2 \ov{A_1}.
\end{aligned}
\right.
\end{equation}
  Denote by $X=A_1 \ov{A_2}$ and let us show that $X\in \R$.
 $$\frac{X+|A_2|^2}{4\big(1-|A_2|^2-X \big)}=X\quad  \Leftrightarrow  \quad     4X^2+(4 |A_2|^2-3  )X+ |A_2|^2=0.$$
This latter equation has real roots if and only if its discriminant 
$$\Delta=(2|A_2|+1)(2|A_2|-1)(2|A_2|+3)(2|A_2|-3) \geq 0\; ,$$
which is the case under the assumption $|A_2|  <\frac12$, hence $X\in \R$. Next, set $x=(A_1+A_2)\ov{A_2}$ and $y=(A_1+A_2)\ov{A_1}$ (which are real numbers) and satisfy
$$x-|A_2|^2=\dis \frac{x}{4(1-x)}=\frac{y}{4(1-y)}=y-|A_1|^2,$$
 by \eqref{4.2}. By injectivity of the function $x \mapsto \frac{x}{4(1-x)}$, this implies $x=y$ and then $|A_1|=|A_2|$. As a consequence $A_1=A_2=Ae^{i\phi}$ or $A_1=-A_2=Ae^{i\phi}$ for some $A\geq 0$ and $\phi \in \R$.

$\bullet$ Assume that $A_1=A_2=Ae^{i\phi}$, then plugging into \eqref{4.1} yields $\dis \frac{A}{2(1-2A^2)}=A$, hence $A=0$ since $A<1/2$.

$\bullet$ Assume that $A_1=-A_2=Ae^{i\phi}$, then we also get $A=0$. \medskip

  {\underline    {Step 3 :  Reduction to the case $B_1=B_2=0$.}} We identify the coefficient of $z$ in the exponential and we get ${B_1=B_2:=\gamma}$.
  
  As a consequence, $U,V$ take the form (recall the notation \eqref{defphia})
  $$U=\sum_{k=n_1}^{m_1}\widetilde{a}_k \phi_k^{\gamma}=R_{-\ov{\gamma}} U_0, \quad V=\sum_{k=n_2}^{m_2}\widetilde{b}_k \phi_k^{\gamma}=R_{-\ov{\gamma}} V_0.$$
  By \eqref{r-1} we get
   \begin{equation*}  
\left\{
\begin{aligned}
& \lambda U+ (\alpha \cdot  \Gamma) U = \Pi (|V|^2 U), \\
& \mu V+ (\alpha \cdot  \Gamma) V = - \Pi (|U|^2 V),
\end{aligned}
\right.
\end{equation*}
if and only if 
   \begin{equation*}  
\left\{
\begin{aligned}
& \lambda_0 U_0+ (\alpha \cdot  \Gamma) U_0 = \Pi (|V_0|^2 U_0),  \\
& \mu_0 V_0+ (\alpha \cdot  \Gamma) V_0 =- \Pi (|U_0|^2 V_0),
\end{aligned}
\right.
\end{equation*}
where $\dis \lambda_0=\lambda-2\Im({\alpha}{\gamma})$ and $\dis \mu_0=\mu-2\Im({\alpha} {\gamma})$. This allows to reduce to the case $\gamma=0$. \medskip

 {\underline    {Step 4 : The case $\alpha=0$.}}
For $\alpha = 0$, the system reads
 \begin{equation*}  
\left\{
\begin{aligned}
& \lambda U_0 = \Pi (|V_0|^2 U_0), \\
& \mu V_0 = - \Pi (|U_0|^2 V_0).
\end{aligned}
\right.
\end{equation*}
We write   $$U_0=\sum_{k=n_1}^{m_1}{a}_k \phi_k, \quad V_0=\sum_{k=n_2}^{m_2}{b}_k \phi_k.$$
With the help of \eqref{pi-phi}, we   identify the highest degrees in the system, and we get $m_1+m_2-n_2=m_1$ and $m_1+m_2-n_1=m_2$, thus $m_1=n_1$ and $m_2=n_2$. So we have $U_0=Ae^{i\theta_1}\phi_{n_1},V_0=Be^{i\theta_2}\phi_{n_2}$ with $A,B \geq 0$, $\theta_1,\theta_2 \in \R$. Next, by computing, one  can obtain 
$$ \lambda = \frac{(n_1+n_2)!}{2^{n_1+n_2+1}\pi n_1! n_2!}B^2, \quad \mu= - \frac{(n_1+n_2)!}{2^{n_1+n_2+1}\pi n_1! n_2!}A^2. $$\medskip

 {\underline    {Step 5 : Case $\alpha \neq 0$, \Bk reduction to the case $\alpha_2=0$  and $\alpha_1 > 0$.}} From now on, we assume that $\alpha \neq 0$. Let us write $(U_0,V_0)=L_\theta(U_1,V_1)$ for some $\theta \in \R$. By~\eqref{r-2} we get that the couple $(U_1,V_1)$ satisfies the system
   \begin{equation*}  
\left\{
\begin{aligned}
& \lambda_0 U_1+\big((e^{i\theta }\alpha) \cdot  \Gamma\big)U_1   = \Pi (|V_1|^2 U_1),  \\
& \mu_0 V_1+ \big((e^{i\theta }\alpha) \cdot  \Gamma\big)V_1  =- \Pi (|U_1|^2 V_1).
\end{aligned}
\right.
\end{equation*}
We have $\alpha=|\alpha|  e^{i \rho}$, then with the choice $\theta =-\rho$ we get the system 
   \begin{equation*}  
\left\{
\begin{aligned}
& \lambda_0 U_1+ |\alpha| \Gamma_1 U_1  = \Pi (|V_1|^2 U_1),  \\
& \mu_0 V_1+ |\alpha| \Gamma_1    V_1   =- \Pi (|U_1|^2 V_1).
\end{aligned}
\right.
\end{equation*}

  {\underline    {Step 6 :  Reduction of the Ansatz.}} We write 
  $$U_1=\sum_{k=n_1}^{m_1}{a}_k \phi_k, \quad V_1=\sum_{k=n_2}^{m_2}{b}_k \phi_k.$$
Then, we compute
$$\Gamma_1 \phi_0=-i \phi_1, \quad \Gamma_1 \phi_n=i(-\sqrt{n+1} \phi_{n+1}+\sqrt{n}\phi_{n-1}), \quad n\geq 1.$$
We identify the highest degrees in the system and get $m_1+m_2-n_2=m_1+1$ and $m_1+m_2-n_1=m_2+1$, thus $m_1=n_1+1$ and $m_2=n_2+1$. As a consequence, $U_1,V_1$ take the form
  $$U_1= {a}_0 \phi_{n_1}+a_1 \phi_{n_1+1}, \quad V_1={b}_0 \phi_{n_2}+b_1 \phi_{n_2+1}.$$
Using the action of the phase rotations $T_{\theta_1,\theta_2 }$, we can restrict to the case $a_0 \geq 0$ and $b_0 \geq 0$.
From \eqref{rot-Q} we necessarily have $M(U_1)=M(V_1)$, and by a rescaling in time, we can assume that $M(U_1)=M(V_1)=1$, thus 
$$|a_0|^2+|a_1|^2=1, \quad |b_0|^2+|b_1|^2=1.$$
By \eqref{rot-P}, there exists $\beta \in \R$ such that $Q_-(U_1,V_1)=i\beta$, and therefore
\begin{equation}\label{Q}
Q_-(U_1,V_1)=a_0\ov{a_1} \sqrt{n_1+1}-b_0\ov{b_1} \sqrt{n_2+1}=i\beta.
\end{equation}

  {\underline    {Step 7 :  Writing and solving the system.}}   By \eqref{pi-phi} we can write the expansion
 \begin{multline*}
  |V_1|^2U_1=\\
 \begin{aligned}
&=a_0|b_0|^2  |\phi_{n_2}|^2 \phi_{n_1}+a_1|b_0|^2  |\phi_{n_2}|^2 \phi_{n_1+1}+a_0b_0 \ov{b_1}   \phi_{n_2} \ov{\phi_{n_2+1}}\phi_{n_1}+a_1b_0 \ov{b_1}   \phi_{n_2} \ov{\phi_{n_2+1}}\phi_{n_1+1}\\ 
&\quad+a_0  \ov{b_0}b_1   \phi_{n_2+1} \ov{\phi_{n_2}}\phi_{n_1}+a_1  \ov{b_0}b_1   \phi_{n_2+1} \ov{\phi_{n_2}}\phi_{n_1+1}+a_0|b_1|^2  |\phi_{n_2+1}|^2 \phi_{n_1}+ a_1|b_1|^2  |\phi_{n_2+1}|^2 \phi_{n_1+1} ,
 \end{aligned}
  \end{multline*}
 and thanks to \eqref{pi-phi} we get
  \begin{multline}
 \Pi( |V_1|^2U_1)=\\
 \begin{aligned}
&=\frac{(n_1+n_2)! a_0b_0 \ov{b_1} }{\pi 2^{n_1+n_2+1}  \sqrt{n_1 ! n_2 ! (n_2+1) ! (n_1-1) ! }} {\bf 1}_{n_1 \geq 1} \phi_{n_1-1}    + \\
&\quad +\frac{1}{\pi 2^{n_1+n_2+1} }\big(     \frac{(n_1+n_2)!a_0|b_0|^2}{n_1!n_2!}    +  \frac{(n_1+n_2+1)!a_1b_0 \ov{b_1} }{2\sqrt{(n_1+1)!n_2! (n_2+1)!n_1!}}+\frac{(n_1+n_2+1)!a_0|b_1|^2}{2{n_1!(n_2+1)! }}  \big)\phi_{n_1} \\
&\quad +\frac{1}{\pi 2^{n_1+n_2+1} }\big(     \frac{(n_1+n_2+1)!a_1|b_0|^2}{2(n_1+1)!n_2!}    +  \frac{(n_1+n_2+1)!a_0 \ov{b_0}b_1 }{2\sqrt{(n_1+1)!n_2! (n_2+1)!n_1!}}+\frac{(n_1+n_2+2)!a_1|b_1|^2}{4{(n_1+1)!(n_2+1)! }}  \big)\phi_{n_1+1} \\
&\quad +\frac{(n_1+n_2+2)! a_1  \ov{b_0}b_1}{\pi 2^{n_1+n_2+3}  \sqrt{(n_1+1) ! (n_2+1) ! n_2  ! (n_1+2) ! }} \phi_{n_1+2}.   
 \end{aligned}
  \end{multline}
  
  $\bullet$ Assume that $n_1 \geq 1$ and $n_2\geq 1$. We identify the coefficients of $\phi_{n_1-1}$ and $\phi_{n_2-1}$
    \begin{equation}\label{syst1}  
\left\{
\begin{aligned}
& \alpha a_0i \sqrt{n_1}  =\frac{(n_1+n_2)! a_0  b_0 \ov{b_1}}{\pi 2^{n_1+n_2+1}  \sqrt{n_1 ! n_2  ! (n_2+1) !  (n_1-1) ! }}    \\
&\alpha b_0i \sqrt{n_2}  =-\frac{(n_1+n_2)! a_0   \ov{a_1}b_0}{\pi 2^{n_1+n_2+1}  \sqrt{n_1 ! n_2  ! (n_1+1) !  (n_2-1) ! }},
\end{aligned}
\right.
\end{equation}
and the coefficients of $\phi_{n_1+2}$ and $\phi_{n_2+2}$
    \begin{equation}  \label{syst2} 
\left\{
\begin{aligned}
& -\alpha a_1i \sqrt{n_1+2}  =\frac{(n_1+n_2+2)! a_1   \ov{b_0}b_1}{\pi 2^{n_1+n_2+3}  \sqrt{(n_1+1) ! (n_2+1)  ! n_2  !  (n_1+2) ! }}    \\
&-\alpha b_1i \sqrt{n_2+2}  =-\frac{(n_1+n_2+2)!     \ov{a_0}a_1 b_1}{\pi 2^{n_1+n_2+3}   \sqrt{(n_1+1) ! (n_2+1)  ! n_1  !  (n_2+2) ! }} .
\end{aligned}
\right.
\end{equation}
We first show that $a_0\neq 0 $, $a_1\neq 0 $, $b_0\neq 0 $, $b_1\neq 0 $. Suppose that one of them equals $0$, for instance $a_0=0$. We get:
$$\alpha b_0i \sqrt{n_2}  = 0  \quad \text{and} \quad \alpha b_1i \sqrt{n_2+2}  = 0 .$$
  Since $\alpha\neq0$, it gives $b_0=b_1=0$. That yields a contradiction with $|b_0|^2+|b_1|^2=1$. We show similarly that $a_1\neq0, b_0\neq 0$, and $b_1\neq0$. 
  
From \eqref{syst1} and \eqref{syst2} we then deduce that 
$$ (n_1+n_2+2) (n_1+n_2+1)=4(n_1+1)(n_1+2)=4(n_2+1)(n_2+2).$$
The last equality implies $n_1=n_2=n$. This in turn implies $2(n+1)(2n+1)=4(n+1)(n+2)$, and this latter equation has no solution.
\medskip

 $\bullet$ Assume that $n_1 \geq 1$ and $n_2=0$. We identify the coefficients of $\phi_{n_1-1}$ and $\phi_{n_1+2}$
    \begin{equation}\label{systn}  
\left\{
\begin{aligned}
& \alpha a_0i \sqrt{n_1}  =\frac{n_1! a_0  b_0 \ov{b_1}}{\pi 2^{n_1+1}  \sqrt{n_1 !   (n_1-1) ! }}  ,  \\
& -\alpha a_1i \sqrt{n_1+2}  =\frac{(n_1+2)! a_1   \ov{b_0}b_1}{\pi 2^{n_1+3}  \sqrt{(n_1+1) !    (n_1+2) ! }},
\end{aligned}
\right.
\end{equation}
and the identification of the coefficients of $\phi_{2}$ gives: 
$$   i \alpha \sqrt{2}b_1=\frac{(n_1+2)! b_1   \ov{a_0}a_1}{\pi 2^{n_1+3}  \sqrt{2(n_1+1) !    n_1 ! }}.$$
First, we show that $a_0 \neq 0$ and $a_1 \neq 0$. Assume that $a_0= 0$, then the previous equation implies $b_1=0$ which in turn implies $a_1=0$. This is a contradiction with  $|a_0|^2+|a_1|^2=1$.   Similarly  we prove that  $a_1 \neq 0$. From  system \eqref{systn} we deduce 
$$  i\alpha =   \frac{     b_0\ov{b_1}}{\pi 2^{n_1+1}   } =\frac{     b_0\ov{b_1}}{\pi 2^{n_1+3}   }, $$
and this equation has no solution since $\alpha\neq 0$. The case $n_1 =0$ and $n_2\geq 1$ is similar.  \medskip

 $\bullet$ Assume that $n_1 =0$ and $n_2=0$. The identification of the coefficients of $\phi_0$, $\phi_1$ and $\phi_2$ yields the system 
     \begin{equation}  \label{46}
\left\{
\begin{aligned}
& \lambda_0 a_0 +i \alpha a_1=\frac1{4\pi } \big(2a_0 |b_0|^2+a_0 |b_1|^2+a_1b_0 \ov{b_1}\big)      \\
& \mu_0  b_0 +i \alpha b_1=-\frac1{4\pi } \big(2b_0 |a_0|^2+b_0 |a_1|^2+a_0\ov{a_1}b_1 \big)  \\
& \lambda_0  a_1 -i \alpha a_0=\frac1{4\pi } \big(a_1|b_0|^2+a_1|b_0|^2+a_0\ov{b_0} b_1  \big)  \\
& \mu_0  b_1 -i \alpha b_0=-\frac1{4\pi } \big(b_1|a_0|^2+b_1|a_1|^2+\ov{a_0} a_1b_0  \big) \\
&   -i \alpha \sqrt{2}a_1=\frac{a_1\ov{b_0} b_1}{4\sqrt{2}\pi } \\
&   i \alpha \sqrt{2}b_1=\frac{\ov{a_0} a_1b_1}{4\sqrt{2}\pi } .
\end{aligned}
\right.
\end{equation}

  Let us first show that $a_0\neq 0 $, $a_1\neq 0 $, $b_0\neq 0 $, and  $b_1\neq 0 $. Suppose that $a_0=0$. We get $ i \alpha \sqrt{2}b_1=0$ so $b_1=0$ since $\alpha\neq0$. Then, we replace $b_1=a_0=0$ in the first line of \eqref{46} and get  $a_1=0$. This yields a contradiction with $|a_0|^2+|a_1|^2=1$. The other cases are similar. \medskip

Then using that 
$|a_0|^2+|a_1|^2=1$ and $|b_0|^2+|b_1|^2=1$, we get the system
  \begin{equation}  \label{47}
\left\{
\begin{aligned}
& \lambda_0 a_0 +i \alpha a_1=\frac1{4\pi } \big(a_0  +a_0 |b_0|^2+a_1b_0 \ov{b_1}\big)      \\
& \mu_0  b_0 +i \alpha b_1=-\frac1{4\pi } \big(b_0+b_0 |a_0|^2+ a_0\ov{a_1}b_1 \big)  \\
& \lambda_0  a_1 -i \alpha a_0=\frac1{4\pi } \big(a_1 +a_0\ov{b_0} b_1  \big)  \\
& \mu_0  b_1 -i \alpha b_0=-\frac1{4\pi } \big(b_1 +\ov{a_0} a_1b_0  \big) \\
&   i \alpha =\frac{{b_0} \ov{b_1}}{8\pi }=-\frac{{a_0} \ov{a_1}}{8\pi } .
\end{aligned}
\right.
\end{equation}
Recall that \eqref{Q} implies $a_0\ov{a_1}  -b_0\ov{b_1}  =i\beta$, and thus 
$${a_0} \ov{a_1}=-i 8\pi \alpha =- {b_0} \ov{b_1}, \quad \beta=-16 \pi \alpha.$$
 The third line of \eqref{47} implies 
     \begin{equation*}
4\pi\lambda_0 a_1 - 4\pi i a_0 \alpha  = a_1+a_0\ov{b_0}b_1  = a_1(1-|a_0|^2)  = a_1|a_1|^2.
\end{equation*}
Recall from step 4 that $\alpha>0$, then $\alpha = \ov{\alpha}=\frac{-i}{8\pi}a_1\ov{a_0}$, then the previous line yields  (using $a_1\neq 0$)  
     \begin{equation}\label{E1}
4\pi\lambda_0 = |a_1|^2+\frac{|a_0|^2}{2}.
\end{equation}
 By the first line of  \eqref{47}, we also have 
\begin{equation*}
4\pi\lambda_0 a_0 +4\pi i \alpha a_1  = a_0+|b_0|^2a_0+a_1\ov{b_1}b_0  = a_0 (1+|b_0|^2-|a_1|^2)  = a_0 (|b_0|^2+|a_0|^2).
\end{equation*}
So, by using the expression of $\alpha$ and simplifying by $a_0\neq0$ :
     \begin{equation}\label{E2}
4\pi\lambda_0 - \frac{|a_1|^2}{2} = |b_0|^2 + |a_0|^2.
\end{equation}
By combining \eqref{E1} and \eqref{E2} one can obtain 
     \begin{equation*}
|a_0|^2+|b_0|^2=\frac{1}{2} \quad \text{and thus} \quad |a_1|^2+|b_1|^2=\frac{3}{2}.
\end{equation*}
From $a_1\ov{a_0}=-b_1\ov{b_0}$ we get $|b_1|^2|b_0|^2=|a_1|^2|a_0|^2$ which reads: 
     \begin{equation*}
|b_1|^2(1-|b_1|^2)  =|a_1|^2(1-|a_1|^2)  = (\frac{3}{2}-|b_1|^2)(|b_1|^2-\frac{1}{2})  = -|b_1|^4+2|b_1|^2-\frac{3}{4}.
\end{equation*}
That is to say $|b_1|^2={3}/{4}$, and we deduce that $|a_1|^2={3}/{4}$ and $|a_0|^2=|b_0|^2={1}/{4}$.
The equation~\eqref{E1} yields 
$\dis \lambda_0 = {7}/{(32\pi)}.$
One can show similarly that 
$\dis \mu_0 = -{7}/{(32\pi)}.$
In  step 6 we showed that we can restrict to the case $a_0 \geq 0$ and $b_0 \geq 0$, so we have 
$$ a_0 = b_0 = \frac{1}{2}, \quad  \alpha = \frac{\sqrt{3}}{32\pi}, \quad a_1=\frac{\sqrt{3}}{2}i,\quad b_1=-\frac{\sqrt{3}}{2}i.$$
As a conclusion, we get 
$\dis U_1=  \frac12 \phi_0 +\frac{\sqrt{3}}2 i  \phi_{1}, \;\;  V_1= \frac12 \phi_0 -\frac{\sqrt{3}}2i   \phi_{1},$
and the general progressive waves are given by  
$$(U,V)=\big(K e^{i\theta_1} R_{-\ov{\gamma}}L_{\theta}(U_1),K e^{i\theta_2} R_{-\ov{\gamma}}L_{\theta}(U_1)\big), \quad K\geq 0, \quad  \theta, \theta_1, \theta_2 \in \R, \quad \gamma \in \C,$$
namely
 \begin{equation*} 
 U= Ke^{i\theta_1} \big(\frac12 \phi_0^\gamma +\frac{\sqrt{3}}2 ie^{i\theta} \phi_{1}^\gamma      \big), \quad 
 V= Ke^{i\theta_2} \big(\frac12 \phi_0^\gamma -\frac{\sqrt{3}}2i e^{i\theta} \phi_{1}^\gamma      \big),
\end{equation*}
with 
$$ \lambda = \frac{K^2}{32\pi}\big(7+2\sqrt{3}\Im(\gamma e^{-i\theta})\big), \quad \mu = \frac{K^2}{32\pi}\big(-7+2\sqrt{3}\Im(\gamma e^{-i\theta})\big), \quad \alpha = \frac{\sqrt{3}}{32\pi}K^2e^{-i\theta}.$$  
Denote by  $Q(U)=\int_{\C}z|U(z)|^2dL(z)$ and $P(U)=\int_{\C}(|z|^2-1)|U(z)|^2dL(z)$. Then we compute 
$$Q(U)= K^2(\ov{\gamma}-i \frac{\sqrt{3}}4 e^{-i\theta}),\quad Q(V)= K^2(\ov{\gamma}+i \frac{\sqrt{3}}4 e^{-i\theta}), $$
$$ P(U)= K^2\big(\frac34+|\gamma|^2+  \frac{\sqrt{3}}2  \Im(\gamma e^{-i\theta})    \big), \quad P(V)=K^2\big(\frac34+|\gamma|^2-  \frac{\sqrt{3}}2  \Im(\gamma e^{-i\theta})    \big),$$
\begin{equation*}
\mathcal{H}(U,V)=K^4\mathcal{H}(U_1,V_1)= \frac{11}{64 \pi}K^4,\quad \mathcal{H}(U,U)=\mathcal{H}(V,V)= \frac{23}{64 \pi}K^4,
\end{equation*}
hence \eqref{eq00}.

\begin{rem}
For both values of $\sigma \in \{-1,1\}$, there exist explicit stationary solutions ($\alpha=0)$ with an infinite number of zeros (see~\cite[Proposition A.1]{GGT}), for instance
\begin{equation}\label{chi}
(U,V)=\Big(\frac{\sinh(\gamma z)}{\sqrt{\pi \sinh(|\gamma |^2)}}\, e^{-\frac{|z|^2}{2}},\frac{\sinh(\gamma z)}{\sqrt{\pi \sinh(|\gamma |^2)}}\, e^{-\frac{|z|^2}{2}}\Big),\quad \forall \gamma \in \C^*,
\end{equation}
with $\lambda=\frac1{4\pi}$ and $\mu=\frac{\sigma}{4\pi}$. We do not know such examples in the case $\alpha \neq 0$.
\end{rem}

 \subsection{Proof of Proposition~\ref{prop-gen}} Assume that $(u,v)$ is a nontrivial progressive wave with initial condition $(U,V) \in L^{2,1}_{\E} \times  L^{2,1}_{\E} $, then  $(U,V)$ satisfies the system
  \begin{equation}\label{sys-0}
\left\{
\begin{aligned}
& \lambda U+ (\alpha \cdot  \Gamma) U = \Pi (|V|^2 U),  \\
& \mu V+ (\alpha \cdot  \Gamma) V = - \Pi (|U|^2 V).
\end{aligned}
\right.
\end{equation}
 \quad $(i)$ and $(ii)$ are consequences of the conservation of $P_{-}$ (see \eqref{rot-P} in the case $\sigma=-1$).

$(iv)$ The result \eqref{u3} is obvious. Assume that $(U, V)$ satisfies~\eqref{sys-0} and let us show \eqref{u1}   : by~\eqref{r-2} we get $\lambda L_\theta U + (\alpha e^{-i\theta}) \cdot \Gamma L_\theta U =\Pi\big( |L_\theta V|^2 L_\theta U\big)$, and similarly for the equation in $V$. To show \eqref{u2}, we use \eqref{r-1} to deduce that  $\big(\lambda -2 \Im(\alpha \ov{\beta}) \big) R_\beta U + (\alpha \cdot \Gamma) R_\beta U =\Pi\big( |R_\beta V|^2 R_\beta U\big)$.

$(v)$ We first assume that $\alpha \in \R$, then  $(U, V)$ satisfies the system
  \begin{equation} \label{sys-1}
\left\{
\begin{aligned}
& \lambda U+ \alpha    \Gamma_1 U = \Pi (|V|^2 U),  \\
& \mu V+ \alpha    \Gamma_1  V = - \Pi (|U|^2 V).
\end{aligned}
\right.
\end{equation}
We set (with $z=x+iy$)
\begin{equation*} 
\dis I(U)=\int_{\C} y   |U(z)|^2dL(z)=- \frac i2\int_{\C} (z-\ov{z})  |U(z)|^2dL(z) ,
\end{equation*}
\begin{equation*} 
\dis J(U)=\int_{\C} y^2   |U(z)|^2dL(z)=-\frac14 \int_{\C} (z-\ov{z})^2 |U(z)|^2dL(z),
\end{equation*}
\begin{equation*} 
\dis K(U)=\int_{\C} xy   |U(z)|^2dL(z)=-\frac{i}4 \int_{\C} (z^2-\ov{z}^2) |U(z)|^2dL(z).
\end{equation*}

$\bullet$ We take the scalar product of the first line of \eqref{sys-1} with $U$ : write $U(z)=f(z)e^{-\frac12|z|^2}  $, then we integrate by parts and get 
  \begin{eqnarray*}  
    \int_{\C}  \ov{U} \Gamma_1 U dL&=&i  \int_{\C} \ov{f(z)} \big(-zf(z)+\partial_z f(z)\big)e^{-|z|^2}dL\\
  & =&-i\int_{\C}\big(z -\ov{z}  \big) |f(z)|^2 e^{-|z|^2} dL\\
  &=&2\int_{\C}y |U(z)|^2dL =  2I(U). 
    \end{eqnarray*}   
 We proceed similarly with $V$ and get the system
 \begin{equation} \label{cond1}
\left\{
\begin{aligned}
& \lambda M(U)+ 2{\alpha}I(U)   = \mathcal{H}(U,V),  \\
& \mu M(V)+ 2{\alpha} I(V)  = -\mathcal{H}(U,V).
\end{aligned}
\right.
\end{equation}

$\bullet$ Now we take the scalar product of the first line of \eqref{sys-1} with $zU$    (and similarly for $V$) : we have
 \begin{equation}  \label{scal1}
  \int_{\C}  \ov{zU} \, \Pi (|V|^2 U)dL= \int_{\C}  \ov{z} |U|^2  |V|^2dL,
  \end{equation}
where we used that $zU \in \E$. Then, by integrating by parts 
  \begin{eqnarray}  
    \int_{\C}  \ov{zU}\, \Gamma_1 U dL &=&i  \int_{\C} \ov{zf(z)} \big(-zf(z)+\partial_z f(z)\big)e^{-|z|^2} dL\nonumber \\
     &=&-i\int_{\C}\big(|z|^2 -(\ov{z})^2  \big) |f(z)|^2 e^{-|z|^2} dL\nonumber \\
     &=&2K(U)-2iJ(U), \label{scal2} 
    \end{eqnarray}   
    hence 
   \begin{equation}  \label{cond2}
\left\{
\begin{aligned}
& \lambda Q(U)+ 2 \alpha K(U)+2i\alpha J(U)  = \int_{\C}  z |U|^2  |V|^2dL,  \\
&  \mu  Q(V)+  2 \alpha K(V)+2i\alpha J(V)   = - \int_{\C}  z |U|^2  |V|^2dL.
\end{aligned}
\right.
\end{equation}

$\bullet$ We take the scalar product of the first line of \eqref{sys-1} with $\ov{z}U$    (and similarly for $V$) : we have
 \begin{eqnarray*}   
  \int_{\C}  z\ov{U} \, \Pi (|V|^2 U)dL&= &\int_{\C}   \Pi (|V|^2 U) \ov{   \Pi (\ov{z} U) }   dL\\
  &=&  \int_{\C}   |V|^2 U\, \ov{  \partial_{z}U }   dL + \frac12\int_{\C}  z |U|^2 |V|^2   dL \\
  &=&  \int_{\C}   |V|^2  {\partial_{\ov{z}}} (  |U|^2 )      dL +  \int_{\C}  z |U|^2 |V|^2   dL,
  \end{eqnarray*}
where we used that $\Pi(\ov{w}v)(z) =  \partial_z \Pi v(z)+ \frac{\ov z}2 \Pi v(z)$. Next, by integrating by parts 
  \begin{eqnarray*}  
    \int_{\C} z \ov{U}\, \Gamma_1 U dL &=&i  \int_{\C} z\ov{f(z)} \big(-zf(z)+\partial_z f(z)\big)e^{-|z|^2} dL\nonumber \\
     &=&i\int_{\C}\big(|z|^2-z^2 -1  \big) |f(z)|^2 e^{-|z|^2} dL\nonumber \\
     &=&2iJ(U)+2K(U)-iM(U), 
    \end{eqnarray*}   
    hence 
   \begin{equation}  \label{cond3}
\left\{
\begin{aligned}
& \lambda Q(U)+ 2 \alpha K(U)+2i\alpha J(U)-i\alpha M(U)  =     \int_{\C}   |V|^2  {\partial_{\ov{z}}} (  |U|^2 )      dL +  \int_{\C}  z |U|^2 |V|^2   dL ,  \\
&  \mu  Q(V)+  2 \alpha K(V)+2i\alpha J(V)-i\alpha M(V)   =  - \int_{\C}   |U|^2  {\partial_{\ov{z}}} (  |V|^2 )      dL -  \int_{\C}  z |U|^2 |V|^2   dL.
\end{aligned}
\right.
\end{equation}  
Therefore, by   \eqref{cond2} and   \eqref{cond3} we get (observe that $M(U)=M(V)$)
 \begin{equation}  \label{def-a}
\alpha =\frac{i}{M(U)}  \int_{\C}   |V|^2  {\partial_{\ov{z}}} (  |U|^2 )      dL:=F(U,V).
 \end{equation}   
The general case $\alpha \in \C$ follows from the identities $F(L_\theta U,L_\theta V)=e^{-i\theta } F( U,  V)$ and $\alpha(L_\theta U,L_\theta V)=e^{-i\theta } \alpha( U,  V)$.

$(iii)$ Assume that $\alpha \in \R$. We take the imaginary part of  \eqref{cond2} and we obtain the system 
 \begin{equation}  \label{cond4}
\left\{
\begin{aligned}
& \lambda I(U) +2\alpha J(U)  = \int_{\C} y   |U(z)|^2|V(z)|^2 dL,  \\
&  \mu   I(V)+ 2\alpha J(V)  = -\int_{\C} y   |U(z)|^2|V(z)|^2 dL.
\end{aligned}
\right.
\end{equation}
We finally solve the system \eqref{cond1}-\eqref{cond4}  and get the formula
\begin{equation} \label{for-a0}
\alpha =- \frac{\mathcal{H}(U,V) \Im Q_-(U,V)}{2\big(\Theta(U)+\Theta(V)\big)}= i \frac{\mathcal{H}(U,V)  Q_-(U,V)}{2\big(\Theta(U)+\Theta(V)\big)},
\end{equation}
 where $\Theta(U)=J(U)M(U)-(I(U))^2$. Observe that $\Theta(U)>0$  for any $U \neq 0$ by the Cauchy-Schwarz inequality. 
 When $\alpha \in \R$, the result then follows from \eqref{for-a0}. For the general case $\alpha \in \C^*$ we can also obtain a similar formula by considering  the action of $L_{\theta}$.

$(vi)$ By \eqref{def-a}, Lemma~\ref{lem.L4} and the Carlen inequality \eqref{hypercontract}, 
 \begin{equation*}
 |\alpha|\leq   \frac{ \|V\|^2_{L^4} \| \partial_z(|U|^2)\|_{L^2}   }{\|U\|^2_{L^2} }   \leq \frac{ \|U\|^2_{L^4}\|V\|^2_{L^4} }{\sqrt{2}\|U\|^2_{L^2} }  \leq \frac{ \|U\|^2_{L^2} }{2\sqrt{2} \pi},
\end{equation*}
hence the result.\medskip

$(vii)$ We make the difference between the two lines in \eqref{cond1} and we get the result for $\alpha \in \R$ (recall that from $(ii)$ we have $\Re\big(  Q_-(U,V)\big)=0$). The general case is obtained using the action of~$L_{\theta}$, as in item $(v)$.
 

 \appendix
 
 \section{Some technical results}\label{appendix}
 
 \subsection{Some commutation relations}
 Recall the formula 
   \begin{equation}\label{ralpha}
 R_{\alpha}u(z) =  u(z+\alpha) e^{\frac{1}{2}(\overline z \alpha - z \overline{\alpha})}.
    \end{equation} 
The infinitesimal generator of this transformation   is obtained thanks to a first order Taylor expansion, and we get, for all $\alpha \in \C$
    \begin{equation}\label{rgamma}
    R_{\alpha}u=e^{-i(\alpha \cdot \Gamma)}u,
       \end{equation}  
 with $\alpha=\alpha_1+i\alpha_2$ and $\alpha \cdot \Gamma := \alpha_1 \Gamma_1 +\alpha_2 \Gamma_2$, where $\Gamma_1$ and $\Gamma_2$ are defined by
  \begin{equation*} 
\Gamma_1=i(-z +\partial_z+\frac{\ov z}2), \qquad \Gamma_2=-(z +\partial_z+\frac{\ov z}2).
\end{equation*}
Notice  that in general  $R_\alpha$ (respectively $\alpha \cdot \Gamma$) do not commute with $R_\beta$  (respectively  $\beta \cdot \Gamma$), see~\eqref{r-00} and~\eqref{r-5} below. \medskip

Similarly, for all $\theta \in \R$
 \begin{equation}\label{propa-L}
 L_{\theta}= e^{i \theta (\frac{H}2-1)}. 
   \end{equation}
We have the following commutation results
 \begin{lem}
Let $\alpha, \beta \in \C$ and $\theta \in \R$, then the   following commutation relations hold true
\begin{eqnarray}
  R_{\alpha}  R_\beta &=& e^{\frac12( {\ov \alpha} \beta-\alpha \ov{\beta} )} R_{\alpha+\beta}    \label{r-00} \\
R_{-\beta}  R_{\alpha}  R_\beta &=& e^{ {\ov \alpha} \beta-\alpha \ov{\beta} } R_{\alpha}    \label{r-5} \\
R_{-\beta} ( \alpha \cdot \Gamma ) R_\beta &=&( \alpha \cdot \Gamma )     +2 \Im(\alpha \ov{\beta})  \label{r-1} \\[5pt]
L_{-\theta} R_{\alpha}L_{\theta}&=&R_{\alpha e^{i\theta}}  \label{r-3} \\
L_{-\theta}   ( \alpha \cdot \Gamma ) L_\theta &= & ( \alpha e^{i\theta}) \cdot \Gamma      \label{r-2} \\[5pt]
R_{-\beta} H R_{\beta} & =&  H +2   \big[   (i \beta) \cdot \Gamma \big]+2 |\beta|^2 \label{r-4} \\
L_{-\theta} H L_{\theta} & =&  H. \label{r-6} 
   \end{eqnarray}
 \end{lem}
 
 \begin{proof}
$\bullet$ The formulas~\eqref{r-00} and~\eqref{r-5} are direct consequences of \eqref{ralpha}. From~\eqref{r-5} we deduce that  for all $t>0$ 
$$R_{-\beta}  \frac{R_{t \alpha} -1}{t} R_\beta =\frac{1}{t} \big(  e^{ t({\ov \alpha} \beta-\alpha \ov{\beta} )} R_{t\alpha}-1\big)= \frac{ e^{ t({\ov \alpha} \beta-\alpha \ov{\beta} )}-1}{t}R_{t \alpha}+  \frac{R_{t \alpha} -1}{t},$$
and \eqref{r-1} follows by letting $t \longrightarrow 0$.

$\bullet$ Similarly, we show \eqref{r-3} by a direct computation and \eqref{r-2} by a derivation argument in $\alpha$.

$\bullet$  From \eqref{r-3} we obtain the identity 
$$R_{-\beta}L_{\theta} R_{\beta}=L_{\theta}  R_{-\beta e^{i\theta}}R_{\beta}= L_{\theta}   R_{ \beta(1-e^{i \theta})}e^{i |\beta|^2 \sin(\theta)},$$
then \eqref{r-4} is obtained by   derivation   in $\theta$. From~\eqref{propa-L} we deduce that $H$ and~$L$ commute, hence~\eqref{r-6}.
    \end{proof}

 \subsection{Smoothing effects in $\widetilde{\mathcal{E}}$} We now state two results which rely on the particular structure of~$\widetilde{\mathcal{E}}$. \medskip

 The first  result concerns  the hypercontractivity estimates. We refer to Carlen~\cite{Carlen} for the bounds with the optimal constants.

\begin{lem}\label{lem.hyp}
Assume that $u \in \widetilde{\mathcal{E}}$, then for all $1 \leq p \leq q \leq +\infty$  there exists $C>0$ such that 
$$\| u \|_{L^q(\C)} \leq C\|u \|_{L^p(\C)}.$$
\end{lem}

\begin{proof}
For $u \in \widetilde{\mathcal{E}}$  we have $u=\Pi u$, namely
$$u(z)=\frac{e^{-\frac{|z|^2}2}}\pi  \int_\mathbb{C} e^{\ov  w z - \frac{|w|^2}{2}} u(w) \,dL(w),$$
and therefore, using that   $| e^{ - \frac{|z|^2}{2}+\ov  w z - \frac{|w|^2}{2}} |=   e^{-\frac{|z-w|^2}2}$, we get
\begin{equation*}
|u(z)| \leq \frac{1}{\pi}   \int_\mathbb{C} e^{-\frac{|z-w|^2}2} |u(w)| \,dL(w)=\big(\psi \star |u|\big)(z),
  \end{equation*}
where $\psi(z)= \frac{1}\pi e^{-|z|^2/2}$. Next, for all $1 \leq p \leq q \leq +\infty$, there exists $r \in [1,+\infty]$ such that $\frac1q+1=\frac1r+\frac1p$, and by the  Young inequality 
$$\| u \|_{L^q(\C)} \leq \|\psi \|_{L^r(\C)}     \|u \|_{L^p(\C)}\leq C\|u \|_{L^p(\C)},$$
which was the claim.
\end{proof}

The second result shows that derivatives of terms of the form $u \ov{v}$ with $u,v \in \wt\E$, can be controlled in $L^p$ spaces, without using Sobolev embeddings.

\begin{lem}\label{lem.deri}
For all $n,m, j, k \geq 0$ there exists $C=C(n,m,j,k)>0$ such that for all  $1\leq p \leq \infty$ and $u,v \in \widetilde{\mathcal{E}}$, 
$$ \big\| z^{n+m}\partial^j_{\ov z}\partial^k_z\big(u\ov{v}\big)\big\|_{L^p(\C)} \leq C\Big(  \| u\|_{L^{2p}(\C)}+ \| z^nu\|_{L^{2p}(\C)} \Big) \Big(    \| v\|_{L^{2p}(\C)}+ \|z^m v\|_{L^{2p}(\C)}\Big).$$
In particular, for all $v \in L^{2p} \cap \wt{\E},$
 \begin{equation}\label{smooth-LP}
  \big\| \partial^j_{\ov z}\partial^k_z\big(|v|^2\big)\big\|_{L^p(\C)} \leq C   \| v\|^2_{L^{2p}(\C)}.
   \end{equation}
\end{lem}

Notice that the previous result together with Lemma~\ref{lem.hyp} implies that for all $p\geq 1$ and $k \geq 0$
 \begin{equation}\label{laplace}
\big\| (-\Delta)^k( |v|^2) \big\|_{L^p(\C)} \leq C   \| v\|^2_{L^{2}(\C)}, \quad \forall \, v \in \E.
 \end{equation}
\begin{proof}
Let us write $u(z)=f(z)e^{-\frac{|z|^2}2}$ and $v(z)=g(z)e^{-\frac{|z|^2}2}$, then since
$$u(z)=\frac{e^{-\frac{|z|^2}2}}\pi  \int_\mathbb{C} e^{\ov  w z - \frac{|w|^2}{2}} u(w) \,dL(w),$$
we have 
$$u(z)\ov{v(z)}=\frac{e^{-|z|^2}}\pi\ov{g(z)} \int_\mathbb{C} e^{\ov  w z - \frac{|w|^2}{2}} u(w) \,dL(w),$$
and by differentiating in $z$, we deduce the formula
\begin{eqnarray*}
\partial^k_z\Big(u(z)\ov{v(z)}\Big)&=&\frac{1}\pi\ov{g(z)} \int_\mathbb{C} (\ov{w}-\ov{z})^ke^{-|z|^2 +\ov  w z - \frac{|w|^2}{2}} u(w) \,dL(w)\\
&=&\Big( \frac{1}{\pi}\int_\mathbb{C} e^{-|z|^2 +  \xi \ov{z} - \frac{|\xi|^2}{2}} \ov{v(\xi)} \,dL(\xi) \Big)\Big(\frac{1}{\pi} \int_\mathbb{C} (\ov{w}-\ov{z})^ke^{\ov  w z - \frac{|w|^2}{2}} u(w) \,dL(w)\Big)\\
&:=& F(z,\ov{z}) G_k(z,\ov{z}).
\end{eqnarray*}
By the Leibniz rule, 
\begin{equation*}
\partial^j_{\ov z}\partial^k_z\Big(u(z)\ov{v(z)}\Big)= \sum_{\ell =0}^{j}  {{j}\choose{\ell }} \partial^{\ell}_{\ov z} F(z,\ov{z}) \partial^{j-\ell}_{\ov z} G_k(z,\ov{z}).
\end{equation*} 

$\bullet$ We have
$$\partial^{\ell}_{\ov z} F(z,\ov{z})  = \frac{1}{\pi}\int_\mathbb{C} ({\xi}-{z})^{\ell} e^{-|z|^2 +  \xi \ov{z} - \frac{|\xi|^2}{2}} \ov{v(\xi)} \,dL(\xi).    $$
Now observe that 
 \begin{equation}\label{module}
  |e^{  \xi \ov z - \frac{|\xi|^2}{2}} |= e^{  \frac{\ov  \xi z+ \xi \ov{z}}2 - \frac{|\xi|^2}{2}} = e^{\frac{|z|^2}2}   e^{-\frac{|z-\xi|^2}2},
   \end{equation}
then for $0 \leq \ell \leq j$
\begin{eqnarray*}
  \Big|\partial^{\ell}_{\ov z} F(z,\ov{z})   \Big| &\leq &\frac{1}{\pi} e^{-\frac{|z|^2}2}  \int_\mathbb{C} |z-\xi|^{\ell}e^{-\frac{|z-\xi|^2}2} |v(\xi)| \,dL(\xi) \nonumber\\
  &\leq &\frac{1}{\pi} e^{-\frac{|z|^2}2}  \int_\mathbb{C} \<z-\xi\>^{j}e^{-\frac{|z-\xi|^2}2} |v(\xi)| \,dL(\xi). 
  \end{eqnarray*}
Then, by the inequality $|z|^m \leq C\big(|z-\xi|^m+|\xi|^m\big)$
 \begin{multline}
|z|^m  \big|\partial^{\ell}_{\ov z} F(z,\ov{z})   \big|  \leq \\
\begin{aligned}
&\leq C  e^{-\frac{|z|^2}2}  \int_\mathbb{C} \<z-\xi\>^{j+m}e^{-\frac{|z-\xi|^2}2} |v(\xi)| \,dL(\xi)+ C  e^{-\frac{|z|^2}2}  \int_\mathbb{C} \<z-\xi\>^{j}e^{-\frac{|z-\xi|^2}2} |\xi^m v(\xi)| \,dL(\xi) \\
 &\leq C  e^{-\frac{|z|^2}2}  \int_\mathbb{C} \<z-\xi\>^{j+m}e^{-\frac{|z-\xi|^2}2} \big( |v(\xi)| +  |\xi^m v(\xi)|\big)  \,dL(\xi).\label{b11}
 \end{aligned}
  \end{multline}

 $\bullet$ For  $0 \leq \ell \leq \min{(j,k)}$ 
$$\partial^{\ell}_{\ov z} G_k(z,\ov{z})  =\frac{(-1)^\ell k !}{\pi (k-\ell)!} \int_\mathbb{C} (\ov{w}-\ov{z})^{k-\ell}e^{\ov  w z - \frac{|w|^2}{2}} u(w) \,dL(w),   $$
and from \eqref{module} we deduce that
\begin{equation*} 
  \Big|\partial^{\ell}_{\ov z} G_k(z,\ov{z})   \Big| \leq Ce^{\frac{|z|^2}2}  \int_\mathbb{C} \<z-w\>^{k}e^{-\frac{|z-w|^2}2} |u(w)| \,dL(w),
  \end{equation*}
  and therefore as in \eqref{b11}
 \begin{equation}\label{b12}
  |z|^n\Big|\partial^{\ell}_{\ov z} G_k(z,\ov{z})   \Big|   \leq Ce^{\frac{|z|^2}2}  \int_\mathbb{C} \<z-w\>^{k+n}e^{-\frac{|z-w|^2}2} \big( |u(w)| +|w^nu(w)|\big)\,dL(w).
  \end{equation}
  
  As a consequence, by \eqref{b11} and \eqref{b12}, 
$$\Big|z^{n+m}\partial^j_{\ov z}\partial^k_z\Big(u(z)\ov{v(z)}\Big)\Big| \leq C(\psi \star \tilde{v})(z)(\psi \star \tilde{u})(z),$$  
  where we have set $\tilde{u}(z)=|u(z)|+|z^nu(z)|$,  $\tilde{v}(z)=|v(z)|+|z^mv(z)|$ and  $\psi(z)=  \<z\>^{\max{(j+m,k+n)}}e^{-\frac{|z|^2}2} \in L^1(\C)$. Finally, we apply the Young inequality, and get
\begin{eqnarray*}
 \big\|z^{n+m} \partial^j_{\ov z}\partial^k_z\Big(u(z)\ov{v(z)}\Big)\big\|_{L^p(\C)} &\leq& C  \big\|  \psi \star \tilde{v}\big\|_{L^{2p}(\C)}  \big\|  \psi \star\tilde{u} \big\|_{L^{2p}(\C)} \\
 & \leq& C \big\|  \psi\big\|^2_{L^{1}(\C)} \big\| \tilde{v}\big\|_{L^{2p}(\C)}\big\| \tilde{u}\big\|_{L^{2p}(\C)}\\
  & \leq& C   \big\|  \tilde{v}\big\|_{L^{2p}(\C)}\big\| \tilde{u}\big\|_{L^{2p}(\C)},
 \end{eqnarray*}
which was to prove.
\end{proof}

By \cite[Theorem 1]{Carlen} we have the following striking identity: for all $u \in \wt{\mathcal{E}}$ and all $r>0$, 
  \begin{equation}\label{carlen2} 
  \int_{\C}  \Big|\partial_z(|u(z)|^r)\Big|^2 dL(z)= \frac{r}4 \int_{\C}   |u(z)|^{2r}  dL(z),
  \end{equation} 
  from which Carlen derives the hypercontractivity estimates \eqref{hypercontract}.  In the present work, we only need the particular case $r=2$ in \eqref{carlen2} for which we provide a simple proof.
  
\begin{lem} \label{lem.L4}
For all $u\in \E$,
  \begin{equation*} 
  \int_{\C}  \Big|\partial_z(|u(z)|^2)\Big|^2 dL(z)= \frac12 \int_{\C}   |u(z)|^4  dL(z).
  \end{equation*} 
   \end{lem}

\begin{proof}
  Write  $u= f e^{-\frac{|z|^2}2} \in \E$, then $|u(z)|^2=f(z)\ov{f(z)}e^{-|z|^2}$,
thus 
$$\partial_z\big(|u(z)|^2\big)=\big(\ov{f(z)}\partial_zf(z)-\ov{z}|f(z)|^2\big)e^{-|z|^2},$$
hence
\begin{eqnarray*} 
  \int_{\C}  \Big|\partial_z(|u(z)|^2)\Big|^2 dL(z)&=& \int_{\C}  |\partial_z f|^2|f|^2e^{-2|z|^2} + \int_{\C}  |z|^2|f|^4e^{-2|z|^2}-2\Re \int_{\C} (z \partial_z f)\ov{f} |f|^2e^{-2|z|^2} \\
&=& \int_{\C}  |\partial_z f|^2|f|^2e^{-2|z|^2} - \int_{\C}  |z|^2|f|^4e^{-2|z|^2}+  \int_{\C}    |f|^4e^{-2|z|^2},
\end{eqnarray*}
by integrating by parts.  Now compute 
\begin{eqnarray*} 
  \int_{\C}  |z|^2|f|^4e^{-2|z|^2}   &=&- \frac12   \int_{\C}\partial_z \big(  e^{-2|z|^2} \big) zf^2 \ov{f^2}\\
&=&\frac 12  \int_{\C}    |f|^4e^{-2|z|^2}+  \int_{\C} zf \partial_z f  \ov{f^2}e^{-2|z|^2} \\
&=&\frac 12  \int_{\C}    |f|^4e^{-2|z|^2}- \frac12   \int_{\C}\partial_{\ov z} \big(  e^{-2|z|^2} \big) f \partial_z f \ov{f^2}\\
&=& \int_{\C}  |\partial_z f|^2|f|^2e^{-2|z|^2} +\frac 12  \int_{\C}    |f|^4e^{-2|z|^2}, 
\end{eqnarray*}
hence the result.
\end{proof}


 \section{Bounds on the Sobolev norms for LLL}\label{AppendB}

We consider the  Lowest Landau Level equation 
 \begin{equation}\label{LLL-A} 
\left\{
\begin{aligned}
&i \partial_t u=\Pi(|u|^2u), \quad   (t,z)\in \R\times \C,\\
&u(0,\cdot)=  u_0 \in \E. 
\end{aligned}
\right.
\end{equation}
Using the results of Lemma~\ref{lem3.1} and Lemma~\ref{lem.deri}, we are able to improve the bounds obtained in~\cite[Theorem~1.2]{GGT}, from $\<t\>^{\frac{k-1}2}$ to $\<t\>^{\frac{k-1}4}$ :
 
\begin{thm}\label{bounds}
Let $k\geq 1$ be an integer and $u_0\in L^{2,k}_\mathcal{E}$. Then there exists a unique solution $u\in \mathcal{C}^{\infty}  \big(\R , L^{2,k}_\mathcal{E}\big)$ to  equation~\eqref{LLL-A} and it satisfies for all $t \in \R$   
 $$
   \begin{aligned}  
&\Vert \langle z\rangle ^ku(t)\Vert _{L^2(\C)} \leq  \big\| \<z\>^ku_0\big\|_{L^2(\C)} \big(1+C_k \|\<z\> u_0\|^2_{L^2(\C)}  |t| \big)^{\frac{k-1}4} &\quad\text{if} &\quad k\geq 3\\
&\Vert \langle z\rangle ^2u(t)\Vert _{L^2(\C)} \leq  \big\| \<z\>^2u_0\big\|_{L^2(\C)} \big(1+C \| u_0\|^2_{L^2(\C)}  |t| \big)^{\frac12} &\quad\text{if} &\quad k=2.
\end{aligned}
$$
 Moreover, if  $\langle z\rangle^3 u_0\in L^2(\C)$, then 
 \begin{equation}\label{z2}
 \Vert \langle z\rangle ^2u(t)\Vert _{L^2(\C)} \leq  \big\| \<z\>^3u_0\big\|_{L^2(\C)} \big(1+C \|\<z\> u_0\|^2_{L^2(\C)}  |t| \big)^{\frac{1}4}.
 \end{equation}
\end{thm}

This result is actually contained in Theorem~\ref{borne-poly} of the present paper: let  $u_0 \in L^{2,k}_\mathcal{E}$, and  $u$ be the solution to \eqref{LLL-A}, then the couple $(u,u)$ is a solution to \eqref{sys-sig} with $\sigma=1$, and  Theorem~\ref{borne-poly} can be applied. \medskip
 
 Observe that, for all $p \geq 1$ and $k \geq 0$,  by \eqref{laplace} we have 
 \begin{equation*} 
\big\| (-\Delta)^k\big( |u(t)|^2\big) \big\|_{L^p(\C)} \leq C   \| u_0\|^2_{L^{2}(\C)}, \quad  \forall\, u_0 \in \E,
 \end{equation*}
which shows that the oscillations of $|u|^2$ are bounded in $L^p$, however the terms  $\Vert \langle z\rangle ^ku(t)\Vert _{L^2(\C)}$ may grow.  \medskip
 
 We end this section with a result which shows in some sense  that the $L^{2,k}_\mathcal{E}$-norm of a solution to~\eqref{LLL-A} may only grow slowly since it can be controlled by oscillations:
 \begin{prop} 
Let $k\geq 1$ be an integer and $u_0\in L^{2,k}_\mathcal{E}$. Then the solution $u$ to~\eqref{LLL-A} satisfies  for all $\alpha \in \R^*$  and all $t\in \R$
$$\Big| \int_{\C}  |z|^{2k}|u(t,z)|^2e^{i \alpha |z|^2}dL(z)\Big| \leq C_k (|\alpha|^{-2k}+|\alpha|^{-k})\|u_0\|^2_{L^2(\C)}.$$
\end{prop}

\begin{proof} We integrate by parts
\begin{eqnarray*}
\int_{\C}  |z|^{2k}|u|^2e^{i \alpha |z|^2}dL&=&(i \alpha)^{-k}\int_{\C}  {\ov z}^{k}|u|^2\partial^k_{\ov z}\big( e^{i \alpha |z|^2}\big)dL\\
&=&(-i \alpha)^{-k}\int_{\C} e^{i \alpha |z|^2}\partial^k_{\ov z}\big( {\ov z}^{k} |u|^2\big)dL\\
&=&(-i \alpha)^{-k}\sum_{j=0}^kC_{jk}\int_{\C} e^{i \alpha |z|^2}{\ov z}^{j}\partial^j_{\ov z}\big(  |u|^2\big)dL\\
&=&(-i \alpha)^{-k}\sum_{j=0}^kC_{jk}(i \alpha)^{-j}\int_{\C}\partial^j_{ z}\big( e^{i \alpha |z|^2}\big)\partial^j_{\ov z}\big(  |u|^2\big)dL\\
&=&(-i \alpha)^{-k}\sum_{j=0}^kC_{jk}(-i \alpha)^{-j}\int_{\C} e^{i \alpha |z|^2} \partial^j_{ z}\partial^j_{\ov z}\big( |u|^2\big)dL.
\end{eqnarray*}
By Lemma~\ref{lem.deri} we have
$$\big|\int_{\C} e^{i \alpha |z|^2} \partial^j_{ z}\partial^j_{\ov z}\big( |u|^2\big)dL\big| \leq  \big\| \partial^j_{ z}\partial^j_{\ov z}\big( |u|^2\big)\big\|_{L^1(\C)}\leq C\|u\|^2_{L^2(\C)}=  C\|u_0\|^2_{L^2(\C)},$$
hence the result.
\end{proof}

\end{document}